\documentclass[12pt,leqno]{article}

\usepackage{amssymb,amsfonts,amsmath,amsthm,amscd,mathrsfs}
\usepackage{PSFigure,psfrag}
\usepackage{graphics,graphicx}
\def\IC{{\mathbb C}}
\def\IE{{\mathbb E}}
\def\IP{{\mathbb P}}
\def\IR{{\mathbb R}}
\def\IN{{\mathbb N}}

\def\IZ{{\mathbb Z}}

\def\n{\noindent}
\def\dsl{\textstyle\sum\limits}

\def\dis{\displaystyle}
\def\o{\omega}
\def\fr{\mbox{\footnotesize $\dis\frac{1}{2}$}}
\def\frvier{\mbox{\footnotesize $\dis\frac{1}{4}$}}
\def\ov{\overline}
\def\ve{\varepsilon}
\def\f{\footnotesize}
\def\r{\rightarrow}
\def\point{{\mbox{\large $.$}}}
\def\wh{\widehat}
\def\wt{\widetilde}

\def\cA{{\cal A}}
\def\cB{{\cal B}}
\def\cC{{\cal C}}
\def\cD{{\cal D}}
\def\cE{{\cal E}}

\def\cT{{\cal T}}
\def\cE{{\cal E}}
\def\cI{{\cal I}}

\def\cK{{\cal K}}
\def\cH{{\cal H}}
\def\cM{{\cal M}}

\def\cP{{\cal P}}
\def\cF{{\cal F}}

\def\cS{{\cal S}}
\def\cG{{\cal G}}

\def\cV{{\cal V}}
\def\cW{{\cal W}}
\def\cY{{\cal Y}}
\def\cZ{{\cal Z}}

\dimendef\dimen=0

\newtheorem{theorem}{Theorem}[section]
\newtheorem{lemma}[theorem]{Lemma}

\newtheorem{remark}[theorem]{Remark}

\begin{document}

\baselineskip14pt
\noindent

\begin{center}
{\bf A LOWER BOUND ON THE CRITICAL PARAMETER OF INTERLACEMENT PERCOLATION IN HIGH DIMENSION}
\end{center}

\vspace{0.5cm}

\begin{center}
Alain-Sol Sznitman
\end{center}


\bigskip
\begin{abstract}
We investigate the percolative properties of the vacant set left by random interlacements on $\IZ^d$, when $d$ is large. A non-negative parameter $u$ controls the density of random interlacements on $\IZ^d$. It is known from \cite{Szni07a}, \cite{SidoSzni09a}, that there is a non-degenerate critical value $u_*$, such that the vacant set at level $u$ percolates when $u < u_*$, and does not percolate when $u > u_*$. Little is known about $u_*$, however, random interlacements on $\IZ^d$, for large $d$, ought to exhibit similarities to random interlacements on a $(2d)$-regular tree, where the corresponding critical parameter can be explicitly computed, see \cite{Teix08b}. We show in this article that $\liminf_d u_* / \log d \ge 1$. This lower bound is in agreement with the above mentioned heuristics.
\end{abstract}

\vspace{7cm}

\noindent
Departement Mathematik  
\\
ETH-Zentrum\\
CH-8092 Z\"urich\\
Switzerland

\newpage

\thispagestyle{empty}
~

\newpage
\setcounter{page}{1}

\setcounter{section}{-1}
\section{Introduction}

Informally, random interlacements describe the microscopic picture left by random walks on large recurrent graphs, which are locally transient, at long times appropriately tuned so that the local picture is non-degenerate, cf.~\cite{Szni09a}, \cite{Wind09}, \cite{Wind08}. 
They are for instance helpful in understanding how random walks can create large separating interfaces, cf. \cite{Szni08b}, \cite{Szni09b}, \cite{CernTeixWind09}. In the case where the local transient model is $\IZ^d$, $d \ge 3$, the interlacement at level $u \ge 0$ is a random subset of $\IZ^d$, which is connected, ergodic under translations, and infinite when $u$ is positive, cf.~\cite{Szni07a}. The density of this set monotonically goes from $0$ to $1$ as $u$ goes from $0$ to $\infty$. Its complement, the vacant set at level $u$, has non-trivial percolative properties. In particular there exists a non-degenerate critical value  $u_*$ in $(0,\infty)$, such that for $u < u_*$, the vacant set at level $u$ has an infinite connected component, and for $u > u_*$, all its connected component are finite, cf.~\cite{Szni07a}, \cite{SidoSzni09a}. Little is known about $u_*$ and its behavior for large $d$. A heuristic paradigm lurking in the background of the present work is that for large $d$, random interlacements on $\IZ^d$ ought to exhibit similarities to random interlacements on $2d$-regular trees, for which an explicit expression of the critical value is available, cf.~\cite{Teix08b}. The main result of this work gives an asymptotic lower bound on $u_*$, which is in agreement with this paradigm.

\medskip
Before discussing our results any further, we first describe the model. We refer to Section 1 for precise definitions. Random interlacements consist of a cloud of paths constituting a Poisson point process on the space of doubly infinite $\IZ^d$-valued trajectories modulo time-shift, tending to infinity at positive and negative infinite times. A non-negative parameter $u$ plays in essence the role of a multiplicative factor of the intensity measure of this Poisson point process. One constructs on a suitable probability space $(\Omega, \cA, \IP)$, see (\ref{1.29}) - (\ref{1.32}), the whole family $\cI^u$, $u \ge 0$, of random interlacements at level $u$, see (\ref{1.36}). These random sets are the traces on $\IZ^d$ of the cloud of trajectories modulo time-shift, with ``labels up to $u$''. The complement $\cV^u$ of $\cI^u$ in $\IZ^d$ is the vacant set at level $u$. The law $Q_u$ on $\{0,1\}^{\IZ^d}$ of the indicator function of $\cV^u$ is characterized by the following identity, cf.~(2.16) of \cite{Szni07a}:
\begin{equation}\label{0.1}
Q_u [Y_x = 1, \;\mbox{for all} \; x \in K] = \exp\{ - u \,{\rm cap}(K)\}, \;\mbox{for all finite} \;K \subset \IZ^d\,,
\end{equation}

\n
where $Y_x$, $x \in \IZ^d$, stand for the canonical coordinates on $\{0,1\}^{\IZ^d}$, and cap$(K)$ for the capacity of $K$, see (\ref{1.25}).

\medskip
It is known from Theorem 3.5 of \cite{Szni07a} and Theorem 3.4 of \cite{SidoSzni09a}, that there is a critical value $u_*$ in $(0,\infty)$ such that:
\begin{equation}\label{0.2}
\begin{array}{rl}
{\rm i)} & \mbox{for $u > u_*$, $\IP$-a.s., all connected components of $\cV^u$ are finite,}
\\[1ex]
{\rm ii)} & \mbox{for $u < u_*$, $\IP$-a.s., there exists an infinite connected component in $\cV^u$.}
\end{array}
\end{equation}

\medskip\n
So far little is known about $u_*$, not even that it diverges as $d$ tends to infinity. As we show in this article, this is indeed the case, and our main result states that:

\begin{theorem}\label{theo0.1}
\begin{equation}\label{0.3}
\liminf\limits_d \;u_* \, / \log d \ge 1\,.
\end{equation}
\end{theorem}

Let us give some comments about Theorem \ref{theo0.1}. It is a natural question whether there is a matching upper bound for (\ref{0.3}), and in fact $\lim_d u_* / \log d = 1$. This is indeed the case and the matching upper bound can be found in Theorem 0.1 of \cite{Szni10a}. A similar asymptotic behavior holds for the critical value attached to the percolation of the vacant set of random interlacements on $2d$-regular trees, as $d$ goes to infinity, cf.~Remark 5.3 of \cite{Teix08b}, and this feature is consistent with the heuristic principle mentioned a the beginning of the Introduction. We refer to Remark \ref{rem4.5} for more on this topic.

\medskip
We will now briefly outline the proof of Theorem \ref{theo0.1}. The general strategy we follow is similar in its broad lines to that of \cite{Gord91}, and to Section 4 of \cite{AlonBenjStac04}, where the asymptotic behavior for large $d$ of the critical probability of Bernoulli site, and bond percolation, on $\IZ^d$ is analyzed. However, in contrast to Bernoulli percolation, random interlacements have a long range dependence, and the implementation of the general strategy in the above mentioned references is substantially different in the present context.

\medskip
The proof of (\ref{0.3}) begins with a reduction step. If $C = \{0,1\}^d \subseteq \IZ^d$, denotes the hypercube, we consider an event $\cG_u$, which roughly speaking, corresponds to the fact that $\cV^u \cap C$ has a ubiquitous component, i.e.~neighboring most sites of $C$, that a similar property holds for the four translates $C + 2y$, where $y$ in $\IZ^2 ( \subset \IZ^d)$ is a neighbor of the origin, and that all these components communicate together. We refer to (\ref{2.3}) for the precise definition. We show in Theorem \ref{theo2.2} that given any non-negative sequence $u(d)$ such that
\begin{equation}\label{0.4}
\limsup\limits_d \;d^3 \,\IP[\cG^c_{u(d)}] < \infty\,,
\end{equation}
it follows that
\begin{equation}\label{0.5}
u_* \ge u(d), \;\mbox{for all but finitely many $d$}\,.
\end{equation}

\medskip\n
In the case of finite range dependence, such statements are typically proved with the help of stochastic domination, cf.~\cite{LiggSchoStac97}, \cite{Pisz96}. In the present context, long range dependence prevents the use of such arguments. We instead prove (\ref{0.5}) with a variation of the renormalization procedure from Section 4 of \cite{Szni07a}.

\medskip
As a result of this reduction step the main Theorem \ref{theo0.1} follows once we show that for small $\ve > 0$,
\begin{equation}\label{0.6}
\mbox{condition (\ref{0.4}) holds when $u(d) \sim (1- 3 \ve) \log d$}\,.
\end{equation}

\n
The factor 3, which appears above, is immaterial, and simply reflects the presence of three steps in the proof of (\ref{0.6}). The first step corresponds to Theorem \ref{theo3.1}. We introduce a certain $u_0 \sim (1- \ve) \log d$, see (\ref{3.2}), and show that when $d$ is large, with overwhelming probability, most sites in $C$ have a neighbor in a substantial component of $\cV^{u_0} \cap C$, with at least $d^{c(\ve)}$ sites, with $c(\ve)$ a large positive constant independent of $d$. The proof of Theorem \ref{theo3.1} is somewhat intricate. Given a site in $C$, we construct a random tree in $C$, rooted at the site, with large depth $c^\prime(\ve)$ and about $c^{\prime\prime}(\ve) d^{\ve}$ descendants at each node. When $d$ is large, this random tree, except maybe for its root, is almost contained in $\cV^{u_0}$. Namely $\cI^{u_0}$ only meets a small fraction of each generation of the tree, see Lemma \ref{lem3.3}. After pruning this random tree, we construct a substantial random sub-tree, which except maybe for its root lies in $\cV^{u_0} \cap C$. This yields Theorem \ref{theo3.1}.

\medskip
Then the proof of (\ref{0.6}) involves two sprinkling operations. These two steps amount to successively replacing $\cV^{u_0}$ with $\cV^{u_1} \supset \cV^{u_0}$ and with $\cV^{u_2} \supset \cV^{u_1}$, where $u_1 \sim (1- 2 \ve) \log d$ and $u_2 = u(d) \sim (1 - 3 \ve) \log d$, as in (\ref{0.6}). The spirit is similar to \cite{AlonBenjStac04}, \cite{Gord91}. The first sprinkling is conducted in Theorem \ref{theo4.2}. With the help of the isoperimetric properties of $C$, it enables us to show that most of the substantial components of $\cV^{u_0} \cap C$ typically merge together at level $u_1$ and create a component of $\cV^{u_1} \cap C$, which is ubiquitous in $C$. The second sprinkling is conducted in Theorem \ref{theo4.4} and ensures that the various ubiquitous components in $\cV^{u_1} \cap (C + 2y)$, where $y$ in $\IZ^2$ is either the origin or a neighbor of the origin, typically communicate together at level $u_2$. The claim (\ref{0.6}) is then established at the end of this procedure, and as explained above, the main Theorem \ref{theo0.1} follows.

\medskip
We will now describe the structure of the article.

\medskip
In Section 1 we introduce notation and recall various useful facts concerning random walks and random interlacements. We develop in Lemma \ref{lem1.1} some estimates on the Green function and on random walks on $\IZ^d$, with $d$ tending to infinity.

\medskip
The main objective of Section 2 is Theorem \ref{theo2.2}, where the reduction step showing that (\ref{0.5}) follows from (\ref{0.4}) is established.

\medskip
In Section 3 we begin with the proof of (\ref{0.6}). We show in Theorem \ref{theo3.1} that with overwhelming probability most points of $C$ are neighbors of substantial components in $\cV^{u_0} \cap C$. An important intermediate step is achieved in Lemma \ref{lem3.3}.

\medskip
The last Section 4 contains two successive sprinkling operations in Theorem \ref{theo4.2} and Theorem \ref{theo4.4}, and completes the proof of (\ref{0.6}). We discuss further extensions and open problems in Remark \ref{rem4.5}.

\medskip
Let us finally explain our convention concerning constants. Throughout the text $c$ or $c^\prime$ denote positive constants, with values changing from place to place. These constants are independent of $d$. In Section 3 and 4, the constants may depend on the parameter $\ve$ from (\ref{3.1}). Numbered constants $c_0, c_1,\dots $ are fixed and refer to the value pertaining to their first appearance in the text. Dependence of constants on additional parameters appears in the notation. For instance $c(k)$ denotes a constant depending on $k$, in Section 1 and 2, and depending on $\ve$ and $k$ in Sections 3 and 4.

\section{Notation and some useful facts}
\setcounter{equation}{0}

The main object of this section is to introduce notation and collect several useful facts concerning Green function on $\IZ^d$ for large $d$, random walks, and random interlacements.

\medskip
We write $\IN = \{0,1,2,\dots \}$ for the set of natural numbers. Given a non-negative real number $a$, we write $[a]$ for the integer part of $a$. We let $|\cdot |_1$ and $| \cdot |_\infty$ respectively stand for the $\ell^1$ and the $\ell^\infty$ distances on $\IZ^d$. We denote with $(e_i)_{1 \le i \le d}$, the canonical basis of $\IR^d$. Unless explicitly stated otherwise, we tacitly assume that $d \ge 3$.

\medskip
We say that $x, x^\prime$ in $\IZ^d$ are neighbors, respectively $*$-neighbors, when $|x-x^\prime|_1 = 1$, respectively $|x-x^\prime |_\infty = 1$. By finite path, respectively finite $*$-path, we mean a finite sequence $x_0,x_1,\dots,x_N$ in $\IZ^d$, with $N \ge 0$, and such that $x_i$ and $x_{i+1}$ are neighbors, respectively $*$-neighbors, for each $0 \le i < N$. We often simply write path, or $*$-path, in place of finite path, or finite $*$-path, when this causes no confusion. We denote with $B_1(x,r)$ and $S_1(x,r)$, the closed $| \cdot |_1$-ball and $| \cdot |_1$-sphere of radius $r \ge 0$ and center $x \in \IZ^d$. We write $B_\infty(x,r)$ and $S_\infty(x,r)$ in the case of the $| \cdot |_\infty$-distance. Given $A,A^\prime$ subsets of $\IZ^d$, we write $A + A^\prime$ for the set of $x + x^\prime$ with $x$ in $A$ and $x^\prime$ in $A^\prime$, as well as $d_1(A,A^\prime) = \inf\{|x-x^\prime|_1$; $x \in A$, $x^\prime \in A^\prime\}$ for the mutual $|\cdot |_1$-distance between $A$ and $A^\prime$. We write $d_\infty(A,A^\prime)$ in the case of the $|\cdot |_\infty$-distance. When $A = \{x\}$ is a singleton, we write $d_1(x,A^\prime)$, resp. $d_\infty (x,A^\prime)$ for simplicity. The notation $U \subset \subset \IZ^d$ means that $U$ is a finite subset of $\IZ^d$. Given a subset $U$ of $\IZ^d$, we write $|U|$ for the cardinality of $U$, as well as $\partial U$, $\partial_{\rm int}\, U, \ov{U}$ for the boundary, the interior boundary, and the closure of $U$:
\begin{equation}\label{1.1}
\begin{split}
\partial U &= \{x \in U^c; \exists x^\prime \in U, |x-x^\prime|_1 = 1\}, 
\\
\partial_{\rm int} \,U &= \{x \in U; \; \exists x^\prime \in U^c, |x-x^\prime|_1 = 1\}, \; \ov{U} = U \cup \partial U\,.
\end{split}
\end{equation}
When $F, U$ are subsets of $\IZ^d$, we write
\begin{equation}\label{1.2}
\partial_F \,U = \partial U \cap F \;\;\mbox{and} \;\; \ov{U}^F = \ov{U} \cap F\,,
\end{equation}
for the relative boundary and the relative closure of $U$ in $F$.

\medskip
We denote with $W_+$ the space of nearest neighbor $\IZ^d$-valued trajectories defined for non-negative times and tending to infinity. We write $\cW_+$, $X_n$, $n \ge 0$, and $\cF_n$, $n \ge 0$, for the canonical $\sigma$-algebra, the canonical process, and the canonical filtration on $W_+$. The canonical shift on $W_+$ is denoted by $\theta_n$, $n \ge 0$, so that $\theta_n(w) (\cdot) = w( \cdot + n)$, for $w \in W_+$. Since $d \ge 3$, the simple random walk on $\IZ^d$ is transient, and for $x \in \IZ^d$, we denote with $P_x$ the restriction of the canonical law of the simple random walk starting at $x$ to the set $W_+$, which has full measure. We write $E_x$ for the corresponding expectation. When $\rho$ is a measure on $\IZ^d$, we denote with $P_\rho$ the measure $\sum_{x \in \IZ^d} \rho(x) P_x$ and with $E_\rho$ for the corresponding expectation. Given $U \subseteq \IZ^d$, we write $H_U = \inf\{n \ge 0; X_n \in U\}$, $\wt{H}_U = \inf\{n \ge 1; X_n \in U\}$, $T_U = \inf\{ n \ge 0; X_n \notin U\}$, for the entrance time in $U$, the hitting time of $U$, and the exit time from $U$. In case of a singleton $\{x\}$, we write $H_x$ and $\wt{H}_x$ for simplicity.

\medskip
The next lemma provides a bound on the exponential moment of the time spent by the simple random walk in a finite set, which we will use in several occasions in the sequel.

\begin{lemma}\label{lem1.1}
Given $K \subset \subset \IZ^d$, and $\lambda \ge 0$, such that $e^\lambda \sup_{x \in K} P_x [\wt{H}_K < \infty] < 1$, one has
\begin{equation}\label{1.3}
\sup\limits_{x \in \IZ^d} E_x [e^{\lambda \sum_{n \ge 0} 1\{X_n \in K\}}] \le 1 + \dis\frac{e^\lambda - 1}{1 - e^\lambda \sup\limits_{x \in K} P_x[\wt{H}_K < \infty]} \,.
\end{equation}
\end{lemma}

\begin{proof}
We assume $K \not= \emptyset$ without loss of generality, and write
\begin{equation*}
\psi(x) = E_x \Big[\exp\Big\{\lambda \dsl_{n \ge 0} 1\{X_n \in K\}\Big\}\Big], \; \mbox{for $x \in \IZ^d$}\,.
\end{equation*}

\n
An application of the strong Markov property shows that $\sup_{x \in \IZ^d} \psi(x) = \sup_{x \in K} \psi(x)$. Then an application of the strong Markov property at time $\wt{H}_K$ shows that for $x \in K$:
\begin{align*}
\psi(x) & \le e^\lambda P_x [\wt{H}_K = \infty] + e^\lambda E_x [\wt{H}_K < \infty, \psi(X_{\wt{H}_K})]
\\
& \le e^\lambda + e^\lambda P_x [\wt{H}_K < \infty] (\sup\limits_K \psi - 1) \,.
\end{align*}

\n
Using a routine approximation argument of $\sum_{n \ge 0} 1\{X_n \in K\}$ by a finite sum, to exclude the possibility that $\sup_K \psi$ is infinite, we find that
\begin{equation}\label{1.4}
\sup\limits_K \psi - 1 \le (e^\lambda - 1) (1 - e^\lambda \,\sup\limits_{x \in K} P_x [\wh{H}_K < \infty])^{-1} ,
\end{equation}
and (\ref{1.3}) follows.
\end{proof}

We denote with $g(\cdot, \cdot)$ the Green function:
\begin{equation}\label{1.5}
g(x,x^\prime) = \dsl_{n \ge 0} \, P_x [X_n = x^\prime], \;\mbox{for $x,x^\prime$ in $\IZ^d$},
\end{equation}
and write
\begin{equation}\label{1.6}
g(x) = g(x,0), \;\mbox{for $x \in \IZ^d$}\,.
\end{equation}

\medskip\n
The Green function is symmetric in its two variables and due to translation invariance $g(x,x^\prime) = g(x - x^\prime) = g(x^\prime - x)$. One has the following useful representation of $g(\cdot)$, see Montroll \cite{Mont56}, (2.10), p.~243:
\begin{equation}\label{1.7}
g(x) = \dis\int^\infty_0 e^{-u} \prod\limits^d_{i=1} I_{x_i} \Big(\mbox{\f $\dis\frac{u}{d}$}\Big) \,du, \;\mbox{for $x = (x_1,\dots,x_d) \in \IZ^d$},
\end{equation}

\n
where for $n \in \IZ$, $I_n(\cdot)$ stands for the modified Bessel function of order $n$, see Olver \cite{Olve74}, p.~60:
\begin{equation}\label{1.8}
I_n(u) = \dis\frac{1}{\pi}  \;\dis\int^\pi_0 e^{u \cos \theta} \cos n \,\theta \, d \,  \theta, \;u \in \IC\,.
\end{equation}

\medskip\n
We record in the next lemma some useful bounds on $g(\cdot)$ that pertain to its behavior in high dimension, close to the origin, cf.~(\ref{1.9}), (\ref{1.10}), and at large distances, cf.~(\ref{1.11}). We then derive a lower bound on the probability that the starting point of the random walk is the point of the trajectory with minimal $|\cdot |_1$-distance to the origin.

\begin{lemma}\label{lem1.2} 
$(d \ge 3)$
\begin{align}
& g(0) = 1 + \mbox{\f $\dis\frac{1}{2d}$} + o\Big(\mbox{\f $\dis\frac{1}{d}$}\Big), \;\mbox{as $d \rightarrow \infty$} \label{1.9}\,.
\\[1ex]
& \sup\limits_{|x|_1 = k} \;g(x) \le c(k)\,d^{-k}, \;\mbox{for $k \ge 1$} \,.\label{1.10}
\\[1ex]
& g(x) \le g(0) \wedge (c\, d / |x|_1)^{\wt{d}-2}, \;
 \mbox{for $x \in \IZ^d$, with $\wt{d} = \frac{d}{2}$ and $d \ge 5$} \,. \label{1.11}
\\[1ex]
& \inf\limits_{|x|_1 = k} \,P_x \big[\wt{H}_{B_1(0,k)} = \infty\big] \ge 1 - \mbox{\f $\dis\frac{c(k)}{d}$}, \;\mbox{for $k \ge 0$}\,. \label{1.12}
\end{align}
\end{lemma}

\begin{proof}
The proof of (\ref{1.9}) can be found in \cite{Mont56}, p.~264. With regard to (\ref{1.10}), we only need to consider the case $d > k$, for the case $d \le k$ then follows by adjusting constants. Using symmetry we can assume that the components $x_i$ of $x$ are non-negative, vanish when $i >k$, and add-up to $k$. This amounts to a finite number of possibilities for $x_i, 1 \le i \le k$, bounded by $c(k)$. We now show that for any such choice of $x_1,\dots,x_k$, and $x = (x_1,\dots,x_k, 0, \dots , 0) \in \IZ^d$, one has:
\begin{equation}\label{1.13}
\limsup\limits_{d \r \infty} \,g(x) / \Big(\dis\frac{k !}{x_1! \dots x_k!} \;(2d)^{-k}\Big) \le 1 \,.
\end{equation}

\n
This is more than enough to deduce (\ref{1.10}). To prove the claim (\ref{1.13}) we note that in view of (\ref{1.7}), one has for $M > 0$:
\begin{equation}\label{1.14}
\begin{split}
g(x)/d & = \dis\int^\infty_0 e^{-du} \,\prod\limits^d_{i=1} \,I_{x_i} (u) du
\\
& \le \dis\int^M_0 e^{-d(u - \log I_0 (u))} \prod\limits^k_{i=1} \;\Big(\mbox{\f $\dis\frac{u}{2}$}\Big)^{x_i} / x_i ! \,du + \dis\int^\infty_M \,(e^{-u} \,I_0(u))^d \,du \,,
\end{split}
\end{equation}

\medskip\n
where we have used the fact that $I_n(u) \le I_0(u)$, for $u \ge 0$, see (\ref{1.8}), as well as the inequality $I_n(u) \le (\frac{u}{2})^n \,I_0(u) / n!$, for $n \ge 0, u \ge 0$,  which stems from the expansion $I_n(u) = (\frac{u}{2})^n \sum_{m \ge 0} \; \frac{(u^2/4)^m}{m! (m + n)!}$, cf.~\cite{Olve74}, p.~60.

\medskip
Since $I_0(u) \sim e^u (2 \pi u)^{-1/2}$, as $u$ tends to infinity, cf.~\cite{Olve74}, p.~83, we can choose $M \ge 1$, such that the last integral in (\ref{1.14}) is bounded by 
\begin{equation}\label{1.15}
\dis\int^\infty_M ( \pi u)^{-d/2} du = \Big(\mbox{\f $\dis\frac{d}{2}$} - 1\Big)^{-1} \pi^{-\frac{d}{2}}\, M^{-(\frac{d}{2} - 1)} \,.
\end{equation}

\n
When $d$ is large, this term decays much faster than the expression dividing $g(x)$ in (\ref{1.13}). We thus restrict our attention to the first integral in the second line of (\ref{1.14}). It equals:
\begin{equation*}
2^{-k} \Big(\mbox{\f $\dis\prod\limits^k_{i=1}$} x_i!\Big)^{-1} \dis\int^M_0 \,e^{-d(u - \log I_0(u))} u^k du \,.
\end{equation*}

\medskip\n
Using the Laplace method, cf.~\cite{Olve74}, p.~81, Theorem 7.1, (where $\mu = 1$, $\lambda = k+1$), we see that the contribution of the neighborhood of $0$ in the above integral is dominant when $d$ tends to infinity, and the whole expression is equivalent to:
\begin{equation*}
2^{-k} \Big(\mbox{\f $\dis\prod\limits^k_{i=1}$} x_i!\Big)^{-1} \,k! \,d^{-k-1}\,.
\end{equation*}

\n
With (\ref{1.14}) the claim (\ref{1.13}) follows, and this proves (\ref{1.10}).

\medskip
We then turn to the proof of (\ref{1.11}). The bound $g(x) \le g(0)$ is classical. We thus only need to show that for $x \not= 0$, $g(x) \le (c \wt{d} / |x|_1)^{\wt{d} - 2}$. From the Carne-Varopoulos bound, cf.~\cite{Carn85}, we know that:
\begin{equation}\label{1.16}
P_0 [X_n = x] \le 2 \exp\Big\{ - \mbox{\f $\dis\frac{|x|_1^2}{2n}$}\Big\}, \;\mbox{for $x$ in $\IZ^d$, and $n \ge 1$} \,.
\end{equation}

\n
From Madras-Slade \cite{MadrSlad93}, p.~380, (A.21), one has the diagonal upper bound:
\begin{equation}\label{1.17}
P_0 [X_n = 0] \le \Big(\mbox{\f $\dis\frac{\pi}{4} \;\frac{d}{n}$}\Big)^{d/2}, \;\;\mbox{for $n \ge 1$} \,.
\end{equation}

\n
Using the Chapman-Kolmogorov equation and reversibility one finds that
\begin{equation}\label{1.18}
P_0 [X_{2n} = 0] = \sup\limits_{x^\prime \in \IZ_d} \,P_0[X_{2n} = x^\prime] \ge \sup\limits_{x^\prime \in \IZ^d} \,P_0 [X_{2n+1} = x^\prime]\,.
\end{equation}

\n
Thus with (\ref{1.16}), (\ref{1.17}), adjusting constants we see that for $n \ge 2$,
\begin{equation}\label{1.19}
P_0 [X_n = x] \le \Big(c\; \mbox{\f $\dis\frac{d}{n}$}\Big)^{d/4} \;\exp\Big\{ - c \;\mbox{\f $\dis\frac{|x|^2_1}{n}$}\Big\}\,,
\end{equation}

\n
and this inequality is readily extended to $n = 1$. As a result we see that for $x \not= 0$,
\begin{equation}\label{1.20}
\begin{split}
g(x) & = \dsl_{n \ge 1} P_0 [ X_n = x] \le \dis\int^\infty_1 \,\Big(c \; \mbox{\f $\dis\frac{d}{u}$}\Big)^{d/4}\; \exp\Big\{ - c \;\mbox{\f $\dis\frac{|x|^2_1}{u}$}\Big\} \,du
\\
& \le (cd)^{d/4} \dis\int^\infty_0  \Big(c \; \mbox{\f $\dis\frac{v}{|x|_1^2}$}\Big)^{d/4} \;\exp\{-v\} \; c \; \mbox{\f $\dis\frac{|x|_1^2}{v^2}$} \;dv 
\\
&\stackrel{d \ge 5}{=} (cd)^{d/4} \,\Gamma\Big(\mbox{\f $\dis\frac{d}{4}$} - 1\Big) (c \,|x|_1)^{-(\frac{d}{2} - 2)} \,.
\end{split}
\end{equation}

\n
Using the bound $d^2 \le c^{d/2}$ and $\Gamma (\frac{d}{4} - 1) \le c(\frac{d}{4})^{d/4}$, see\cite{Olve74}, p.~88, the claim (\ref{1.11}) follows.

\medskip
Let us finally prove (\ref{1.12}). The case $k = 0$ is a direct consequence of (\ref{1.9}). We thus restrict our attention to the case $k \ge 1$. Adjusting constants and using symmetry, we assume without loss of generality, that $d > 10 k$, and the components $x_i$ of $x$, are non-negative, vanish when $i > k$, and add-up to $k$. We then write
\begin{equation}\label{1.21}
\begin{array}{l}
P_x [\wt{H}_{B_1(0,k)} = \infty] \ge P_x[|X_{2k}|_1 = 3k] \;\inf\limits_{|z|_1 = 3k} \,P_z[H_{B_1(0,k)} = \infty] \ge 
\\[1ex]
P_x [|X_{2k}|_1 = 3k] (1 - \sup\limits_{|z|_1 = 3k} \,P_z[H_{B_1(0,k)} < \infty])\,.
\end{array}
\end{equation}
With (\ref{1.10}) it follows that
\begin{equation}\label{1.22}
\begin{split}
\sup\limits_{|z|_1 = 3k} P_z [H_{B_1(0,k)} < \infty] & \le \sup\limits_{|z|_1 = 3k, |z^\prime|_1 \le k} g(z^\prime - z) |B_1(0,k)| 
\\
&\hspace{-1ex} \stackrel{(\ref{1.10})}{\le} c(k) \,d^{-2k} (1 + 2d + \dots + (2d)^k)
\le c^\prime(k) \,d^{-k}\,.
\end{split}
\end{equation}

\medskip\n
On the other hand we also have the lower bound
\begin{equation}\label{1.23}
P_x [|X_{2k}|_1 = 3k] \ge \Big(1 - \mbox{\f $\dis\frac{k}{d}$}\Big) \Big(1 - \mbox{\f $\dis\frac{k+1}{d}$}\Big) \dots \Big(1 - \mbox{\f $\dis\frac{3k-1}{d}$}\Big) \ge \Big(1 - \mbox{\f $\dis\frac{3k}{d}$}\Big)^{2k} \ge 1 -\mbox{\f $ \dis\frac{c(k)}{d}$} \;.
\end{equation}

\medskip\n
Inserting (\ref{1.22}), (\ref{1.23}) in (\ref{1.21}) proves (\ref{1.12}) for $d \ge c(k)$. Adjusting constants, the general case $d \ge 3$ follows.
\end{proof}

\begin{remark}\label{rem1.3}  ~

\medskip\n
{\rm  1) One can provide in a straightforward manner a companion lower bound to (\ref{1.13}). Indeed in the notation of (\ref{1.13}) one finds that
\begin{equation*}
g(x) \ge P_0 [X_k = x] = \dis\frac{k !}{x_1! \dots x_k!} \;(2d)^{-k} \,,
\end{equation*}

\medskip\n
so that in fact for $x = (x_1,\dots,x_d) \in S_1 (0,k)$, with $x_i \ge 0$ and $x_i = 0$, for $i > k$, one has:
\begin{equation*}
\lim\limits_{d \r \infty} \;g(x) / \Big(\dis\frac{k!}{x_1! \dots x_k!} \;(2d)^{-k}\Big) = 1\,.
\end{equation*}

\medskip\n
Thus for $x$ as above, when $d$ tends to infinity, the ``main contribution'' to $g(x) = g(0,x)$ in (\ref{1.5}), comes from the term $P_0[X_k = x]$.

\bigskip\n
2) The proof of (\ref{1.11}) can easily be modified so that, adjusting constants, one obtains a similar bound to (\ref{1.11}) where $\wt{d} = \frac{d}{2}$ is replaced with $\wt{d}$ arbitrarily close to $d$. The bound stated in (\ref{1.11}) will however suffice for our purpose.} \hfill $\square$
\end{remark}

We now recall some facts concerning equilibrium measure and capacity. Given $K \subset \subset \IZ^d$, we write $e_K$ for the equilibrium measure of $K$, and cap$(K)$ for its total mass, the capacity of $K$:
\begin{align}
&e_K(x) = P_x [\wt{H}_K = \infty], \;\mbox{for $x \in K$, and $e_K(x) = 0$, for $x \notin K$,} \label{1.24}
\\
&{\rm cap} (K) = e_K (\IZ^d) = \dsl_{x \in K} \,P_x [\wt{H}_K = \infty] \,.\label{1.25}
\end{align}
 
\medskip\n
The capacity is known to be subadditive in the sense that ${\rm cap}(K \cup K^\prime) \le {\rm cap}(K) + {\rm cap}(K^\prime)$, for $K, K^\prime \subset \subset \IZ^d$, (this readily follows from (\ref{1.25})). One also has, cf.~(1.62) of \cite{Szni07a},
\begin{equation}\label{1.26}
{\rm cap}(\{x\}) = g(0)^{-1}, \; \mbox{for $x \in \IZ^d$},
\end{equation}

\n
and the probability to enter $K \subset \subset \IZ^d$ can be expressed as
\begin{equation}\label{1.27}
P_x[H_K < \infty] = \dsl_{x^\prime \in K} \,g (x,x^\prime) \,e_K(x^\prime) \,.
\end{equation}

\n
Further one has the bounds, see (\ref{1.9}) of \cite{Szni07a}
\begin{equation}\label{1.28}
\begin{split}
\dsl_{x^\prime \in K} g(x,x^\prime) / (\sup\limits_{z \in K} \;\dsl_{x^\prime \in K} \, g(z,x^\prime)\Big) & \le P_x [H_K < \infty]  
\\[-2ex]
& \le \dsl_{x^\prime \in K} g(x,x^\prime) / \Big(\inf\limits_{z \in K} \;\dsl_{x^\prime \in K} g(z,x^\prime)\Big)\,.
\end{split}
\end{equation}

\medskip\n
We then turn to the description of random interlacements. We refer to Section 1 of \cite{Szni07a} for more details. We denote by $W$ the space of doubly infinite nearest neighbor $\IZ^d$-valued trajectories, which tend to infinity at positive and negative infinite times and by $W^*$ the space of equivalence classes of trajectories in $W$ modulo time-shift. We let $\pi^*$ stand for the canonical map from $W$ into $W^*$. We write $\cW$ for the canonical $\sigma$-algebra on $W$ generated by the canonical coordinates $X_n$, $n \in \IZ$, and we endow $W^*$ with $\cW^* = \{A \subseteq W^*; (\pi^*)^{-1}(A) \in\cW\}$, the largest $\sigma$-algebra on $W^*$ for which $\pi^*: (W, \cW) \rightarrow (W^*, \cW^*)$ is measurable. The canonical probability space $(\Omega, \cA, \IP)$ for random interlacements is defined as follows.

\medskip
We consider the space of point measures on $W^* \times \IR_+$
\begin{equation}\label{1.29}
\begin{split}
\Omega = \Big\{\omega = \dsl_{i \ge 0} \,\delta_{(w_i^*,u_i)}, &\; \mbox{with} \; (w_i^*,u_i) \in W^* 
\times \IR_+, \;\mbox{for} \; i \ge 0, \;\mbox{and}
\\[-2ex]
&\;\omega(W^*_K \times [0,u]) < \infty, \;\mbox{for any} \;K \subset \subset \IZ^d \;\mbox{and} \; u \ge 0\Big\}\,,
\end{split}
\end{equation}

\n
where for $K \subset \subset \IZ^d$, $W^*_K \subseteq W^*$ stands for the set of trajectories modulo time-shift that enter $K$:
\begin{equation}\label{1.30}
\mbox{$W^*_K = \pi^*(W_K)$, with $W_K = \{w \in W$; $X_n(w) \in K$, for some $n \in \IZ$}\}\,.
\end{equation}

\n
We let $\cA$ stand for the $\sigma$-algebra on $\Omega$, which is generated by the evaluation maps $\omega \r \omega(D)$, where $D$ runs over the product $\sigma$-algebra $\cW^* \times \cB(\IR_+)$. We denote with $\IP$ the probability on $(\Omega, \cA)$, which is the Poisson measure with intensity $\nu(d w^*)du$ giving finite mass to the sets $W^*_K \times [0,u]$, for $K \subset \subset \IZ^d$, $u \ge 0$, with $\nu$ the unique $\sigma$-finite measure on $(W^*, \cW^*)$ such that for any $K \subset \subset \IZ^d$, cf.~Theorem 1.1 of \cite{Szni07a}:
\begin{equation}\label{1.31}
1_{W^*_K}\, \nu = \pi^* \circ Q_K\,,
\end{equation}

\n
with $Q_K$ the finite measure on $W^0_K$, the subset of $W_K$ of trajectories that enter $K$ for the first time at time $0$, such that for $A, B \in \cW_+$, $x \in \IZ^d$, in the notation of (\ref{1.24}):
\begin{equation}\label{1.32}
Q_K [(X_{-n})_{n \ge 0} \in A, \;X_0 = x, (X_n)_{n \ge 0} \in B] = P_x[A | \wt{H}_K = \infty] \,e_K (x) \,P_x [B]\,.
\end{equation}

\n
As a direct consequence of (\ref{1.30}) - (\ref{1.32}), we see that for $K, K^\prime \subset \subset \IZ^d$,
\begin{equation}\label{1.33}
\begin{array}{l}
\nu(W^*_K \cap W^*_{K^\prime}) \le P_{e_K} [H_{K^\prime} < \infty]  + P_{e_{K^\prime}} [H_K < \infty] \stackrel{(\ref{1.27})}{=} 2 \cE (K,K^\prime), \;\mbox{with}
\\
\cE(K,K^\prime) \stackrel{\rm def}{=} \dsl_{x \in K, x^\prime \in K^\prime} e_K(x) \,g(x,x^\prime) \,e_{K^\prime}(x^\prime), \, \mbox{the ``mutual energy of $K$ and $K^\prime$''}\,.
\end{array}
\end{equation}

\n
For $K \subset \subset \IZ^d$, $u \ge 0$, one defines on $(\Omega, \cA)$ the following random variable, valued in the set of finite point measures on $(W_+, \cW_+)$:
\begin{equation}\label{1.34}
\mu_{K,u} (\omega) = \dsl_{i \ge 0} \,\delta_{(w_i^*)^{K,+}} \;1\{w_i^* \in W^*_K, u_i \le u\}, \;\mbox{for} \;\omega = \dsl_{i \ge 0} \,\delta_{(w_i^*,u_i)} \in \Omega\,,
\end{equation}

\n
where for $w^*$ in $W^*_K$, $(w^*)^{K,+}$ stands for the trajectory in $W_+$, which follows step by step $w^*$ from the time it first enters $K$. It is shown in Proposition 1.3 of \cite{Szni07a} that for $K \subset \subset \IZ^d$:
\begin{equation}\label{1.35}
\mbox{$\mu_{K,u}$ is a Poisson point process on $(W_+, \cW_+)$ with intensity measure $u \, P_{e_K}$}\,.
\end{equation}

\n
Given $\omega \in \Omega$, the interlacement at level $u \ge 0$ is the subset of $\IZ^d$ defined as:
\begin{equation}\label{1.36}
\begin{split}
\cI^u(\omega) & = \bigcup\limits_{u_i \le u} {\rm range} (w_i^*), \; \mbox{if $\omega = \dsl_{i \ge 0} \,\delta_{(w_i^*, u_i)}$},
\\
& = \bigcup\limits_{K \subset \subset \IZ^d} \; \bigcup\limits_{w \in {\rm Supp} \, \mu_{K,u}(\omega)} \,w(\IN)\,,
\end{split}
\end{equation}

\n
where for $w^* \in W^*$, range$(w^*) = w(\IZ)$, for any $w \in W$, with $\pi^*(w) = w^*$, and the notation ${\rm Supp} \,\mu_{K,u}(\omega)$ refers to the support of the point measure $\mu_{K,u}(\omega)$. The vacant set a level $u$ is then defined as
\begin{equation}\label{1.37}
\cV^u(\omega) = \IZ^d \backslash \cI^u(\omega), \;\mbox{for} \; \omega \in \Omega, \,u \ge 0\,.
\end{equation}
One then sees, cf.~(1.54) of \cite{Szni07a}, that
\begin{equation}\label{1.38}
\cI^u(\omega) \cap K = \bigcup\limits_{w \in {\rm Supp} \,\mu_{K,u}(\omega)} w(\IN) \cap K, \; \mbox{for} \; K \subset \subset \IZ^d, \;u \ge 0, \; \omega \in \Omega \,.
\end{equation}
With (\ref{1.35}) it readily follows that
\begin{equation}\label{1.39}
\IP[\cV^u \supseteq K] = \exp\{ - u \,{\rm cap} (K)\}, \;\mbox{for all $K \subset \subset \IZ^d$},
\end{equation}

\n
which yields the characterization (\ref{0.1}) of the law $Q_u$ on $\{0,1\}^{\IZ^d}$ of the indicator function of $\cV^u(\omega)$, see also Remark 2.2 2) of \cite{Szni07a}.

\medskip
This concludes Section 1 and the above brief account of various useful facts that will be used for the proof of Theorem \ref{theo0.1} in the next three sections.

\section{From local to global}
\setcounter{equation}{0}

In this section we develop a reduction step for the proof of Theorem \ref{theo0.1}. In essence, the main result of this section, Theorem \ref{theo2.2}, shows that, when the level $u$ is such that with high $\IP$-probability $\cV^u$ induces ubiquitous connecting components in the hypercube $\{0,1\}^d$ and its four neighboring hypercubes $2 y + \{0,1\}^d$, with $|y|_1 = 1$ and $y$ in $\IZ^2$, then $\cV^u$ percolates. Similar statements are well-known in the case of finite range dependence, see Liggett-Schonmann-Stacey \cite{LiggSchoStac97}, Pisztora \cite{Pisz96}. The difficulty we encounter here, stems from the long range dependence of random interlacements. We use a variation on the renormalization method of Section 4 in \cite{Szni07a} to cope with this problem.

\medskip
We tacitly identify $\IZ^2$ with the collection of sites in $\IZ^d$ of the form $y = (y_1,y_2,0,\dots,0)$. For $y$ in $\IZ^2$, we consider the hypercube translated at $2y$:
\begin{equation}\label{2.1}
C_y = C + 2y, \;\mbox{where} \; C = \{0,1\}^d\,.
\end{equation}

\n
The notion of {\it ubiquitous component} of $\cV^u \cap C_y$ relevant for our purpose appears in the next lemma. We recall (\ref{1.2}) for the notation.

\begin{lemma}\label{lem2.1}
When $d \ge c$, any subset $V$ of $C$ contains at most one connected component $U$ of $V$ such that 
\begin{equation}\label{2.2}
|\ov{U}\,^C| \ge (1-d^{-2}) \,|C|\,,
\end{equation}

\n
and any other connected component $U^\prime$ of $V$ satisfies $|U| > |U^\prime |$.
\end{lemma}

\begin{proof}
Let $U^\prime$ be another component of $V$ satisfying (\ref{2.2}). Since $|\partial_C U^\prime| \le d  |U^\prime|$, we find that $|U^\prime | \ge (d+1)^{-1} (1 - d^{-2}) \,|C| > d^{-2} |C| \ge |C \backslash \ov{U}|$, when $d \ge c$. As a result $U^\prime$ meets $\ov{U}$, so that $U$ and $U^\prime$ coincide, a contradiction. The same reasoning shows that any other component $U^\prime$ of $V$ satisfies $|U^\prime | < |U|$.
\end{proof}

Given $y$ in $\IZ^2$ and $u \ge 0$, we introduce the local ``good'' event for the ``neighborhood'' of $C_y$:
\begin{equation}\label{2.3}
\begin{split}
\cG_{y,u} = &\; \Big\{\omega \in \Omega; \;\mbox{for each $y^\prime \in \IZ^2$ with $|y^\prime - y|_1 \le 1$, $\cV^u (\omega) \cap C_{y^\prime}$ contains a}
\\
&\; \mbox{connected component $\cC_y$ with $|\overline{\cC}_y \cap C_y| \ge (1-d^{-2})|C_y|$, and these}
\\
&\;\mbox{components are connected in $\cV^u(\omega) \cap \Big(\bigcup\limits_{|y^\prime - y|_1 \le 1} C_{y^\prime}\Big)\Big\}$}\,.
\end{split}
\end{equation}

\n
When $y = 0$, we write $\cG_u$ in place of $\cG_{y = 0,u}$ for simplicity. When $d \ge c$, with $c$ as in Lemma \ref{lem2.1}, and $\o$ is such that $\{y \in \IZ^2; \o \in \cG_{y,u}\}$ has an infinite connected component it follows that $\cV^u(\omega)$ has an infinite connected component in $\cV^u(\o) \cap (\bigcup_{y \in \IZ^2} C_y)$. In particular the percolation of the set of $y$ in $\IZ^2$ where $\cG_{y,u}$ occurs ensures percolation of $\cV^u$. We will rely on this fact to prove the next theorem, which is the main result of this section.

\begin{theorem}\label{theo2.2}
For any non-negative sequence $u(d)$, $d \ge 3$, such that
\begin{equation}\label{2.4}
\limsup\limits_d \;d^3 \; \IP[\cG^c_{u(d)}] < \infty\,,
\end{equation}
one has
\begin{equation}\label{2.5}
u_*(d) \ge u(d), \;\mbox{for all but finitely many $d$}\,.
\end{equation}
\end{theorem}

\begin{proof}
We first observe that for values of $d$ for which $u(d) \ge d$ holds, one has 
\begin{equation*}
\IP [\cG^c_{u(d)}] \ge \IP[\cV^{u(d)} \cap C = \emptyset] \stackrel{(\ref{1.26}),(\ref{1.39})}{\ge} 1 - |C| \,e^{-\frac{u(d)}{g(0)}} \ge 1 - 2^d \,e^{-\frac{d}{g(0)}} \stackrel{\rm (\ref{1.9})}{\longrightarrow} 1, \;\mbox{as $d \r \infty$} \,.
\end{equation*}

\n
In view of (\ref{2.4}) it follows that $u(d) \ge d$ occurs at most finitely many times. As a result we can assume without loss of generality that
\begin{equation}\label{2.6}
u(d) \le d, \; \mbox{for all $d \ge 3$} \,.
\end{equation}

\n
From now on we write $u$ in place of $u(d)$ for simplicity. From the discussion below (\ref{2.3}) and (\ref{0.2}), see also Remark \ref{rem2.4}, we see that the claim (\ref{2.5}) will follow once we show that:
\begin{equation}\label{2.7}
\mbox{for $d \ge d_0$, $\IP[\{y \in \IZ^2; \omega \in \cG_{y,u}\}$ has an infinite connected component$] > 0$}\,.
\end{equation}

\n
We will prove (\ref{2.7}) with an adaptation of the renormalization scheme of Section 4 of \cite{Szni07a}. We thus consider

\vspace{-4ex}
\begin{equation}\label{2.8}
a = \mbox{\f $\dis\frac{1}{100}$}, \;\;L_0 > 1, \; \mbox{integer},
\end{equation}
and introduce a sequence of length scales $L_n$, $n \ge 0$, via:
\begin{equation}\label{2.9}
L_{n+1} = \ell_n \,L_n, \; \ell_n = 100 [L^a_n] \;(\ge L_n^a), \;\mbox{(so that $L_n \ge L_0^{(1+a)^n}$, for $n \ge 0$)}\,.
\end{equation}

\n
We organize the collection $C_y, y \in \IZ^2$, in a hierarchical way. With this in mind, we let $L_0$ stand for the bottom scale, and $L_1 < L_2 < \dots$ stand for coarser and coarser scales. We define the set of labels at level $n \ge 0$:
\begin{equation}\label{2.10}
J_n = \{n\} \times \IZ^2\,.
\end{equation}

\n
To $m = (n,i) \in J_n$, we attach the $\IZ^2$-boxes
\begin{equation}\label{2.11}
\cD_m = [0,L_n)^2 + i\,L_n \subseteq \wt{\cD}_m = [-L_n, 2 L_n)^2 + i\,L_n \subseteq \IZ^2\,,
\end{equation}

\medskip\n
which in turn are used to define the $\IZ^d$-boxes, cf.~(\ref{2.1}),
\begin{equation}\label{2.12}
D_m = \bigcup\limits_{y \in \cD_m} C_y \subseteq \wt{D}_m = \bigcup\limits_{y \in \wt{\cD}_m} C_y\,.
\end{equation}
We also introduce a set of sites in $\wt{\cD}_m$ located near its boundary:
\begin{equation}\label{2.13}
\wt{\cV}_m = \partial_{\rm int}  [-L_n + 1, \,2L_n - 1)^2 + i \,L_n\,,
\end{equation}

\medskip\n
as well as the event, (see (\ref{2.3}) for the notation),
\begin{equation}\label{2.14}
\begin{split}
B_m  =  \{ & \omega \in \Omega; \;\mbox{there is a $*$-path in $\IZ^2$ from $\cD_m$ to $\wt{\cV}_m$, such that for}
\\
& \mbox{each $y$ in the path, $\omega \in \cG^c_{y,u}\}$}\,.
\end{split}
\end{equation}

\n
An important step in the proof of (\ref{2.7}) is to show that when $d$ is chosen large, and $L_0$ linearly growing with $d$, cf.~(\ref{2.37}), there is a rapid decay to zero of the probability
\begin{equation}\label{2.15}
q_n = \IP[B_m], \; \mbox{with $m \in J_n$ arbitrary},
\end{equation}
(due to translation invariance $q_n$ is well-defined).

\medskip
We will set-up a recurrence relation to bound $q_{n+1}$ in terms of $q_n$, for $n \ge 0$. For this purpose given $m \in J_{n+1}$, $n \ge 0$, we consider the set of labels at level $n$, for which the corresponding $\IZ^2$-box is ``at the boundary of $\cD_m$'':
\begin{equation}\label{2.16}
\cK_1 = \{\ov{m} \in J_n; \;\cD_{\ov{m}} \subseteq \cD_m \; \mbox{and some point of $\cD_{\ov{m}}$ neighbors $\IZ^2 \backslash \cD_m\}$}\,,
\end{equation}

\n
as well as the collection of labels at level $n$, for which the corresponding $\IZ^2$-box  contains some point at $|\cdot |_\infty$-distance $\frac{1}{2} \,L_{n+1}$ of $\cD_m$:
\begin{equation}\label{2.17}
\cK_2 = \Big\{\ov{m} \in J_n; \;\cD_{\ov m} \cap \{z \in \IZ^2; \;d_\infty (z,\cD_m) = \fr  \;L_{n+1}\} \not= \emptyset\Big\}\,.
\end{equation}

\psfragscanon
\begin{center}
\includegraphics[width=17cm]{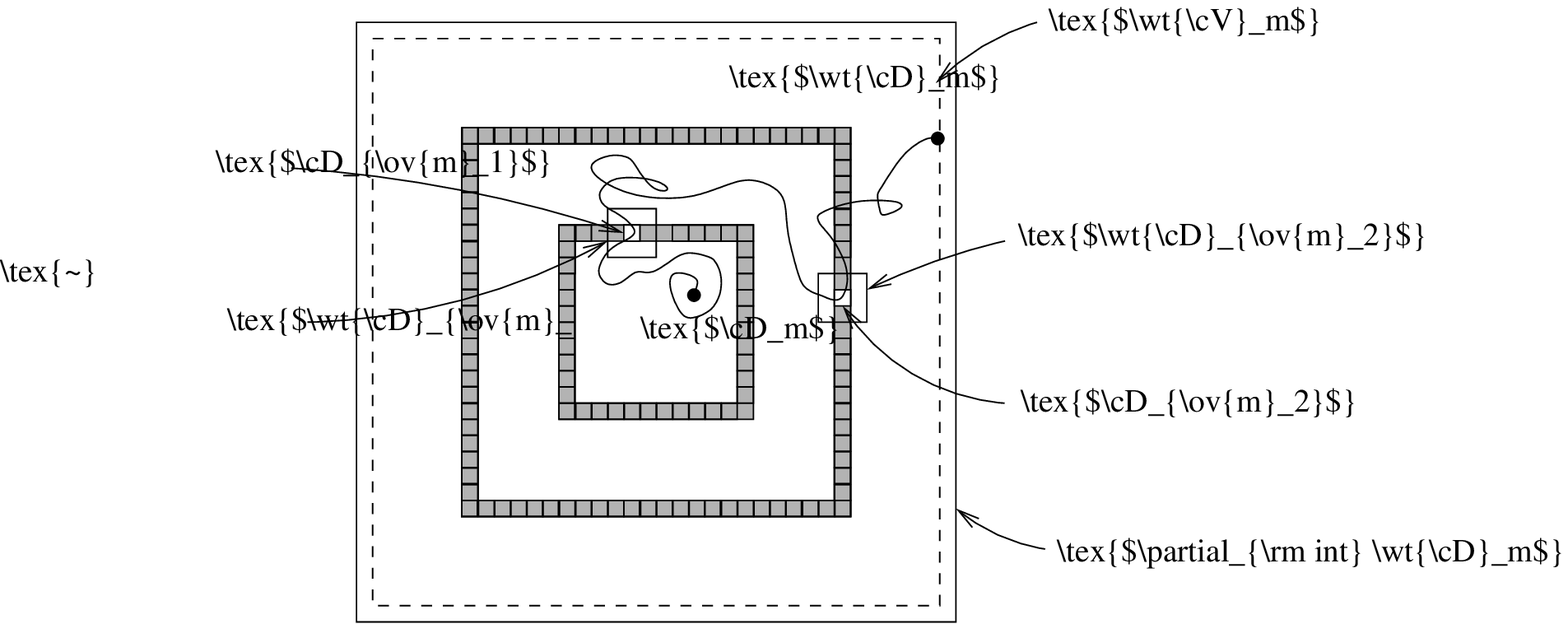}
\end{center}

\medskip
\begin{center}
\hspace{-10ex} \begin{tabular}{ll}
Fig.~1: & A schematic illustration of the event $B_m$. 
\\
&At each site $y$ in the the range of the $*$-path, 
\\
&which is drawn, $\cG^c_{y,u}$ occurs.
\end{tabular}
\end{center}

\medskip\n
Any $*$-path in $\IZ^2$ originating in $\cD_m$ and ending in $\wt{\cV}_m$ must go through some $\cD_{\ov{m}_1}$, $\ov{m}_1 \in \cK_1$, reach $\wt{\cV}_{\ov{m}_1}$, and then go through some $\cD_{\ov{m}_2}, \ov{m}_2 \in \cK_2$, and reach $\wt{\cV}_{\ov{m}_2}$. As a result we see that for $n \ge 0$ and $m \in J_{n+1}$ arbitrary, we have
\begin{equation}\label{2.18}
q_{n+1} \le \dsl_{\ov{m}_1 \in \cK_1, \ov{m}_2 \in \cK_2} \;\IP [B_{\ov{m}_1} \cap B_{\ov{m}_2}] \le c_0 \,\ell^2_n \;\sup\limits_{\ov{m}_1 \in \cK_1, \ov{m}_2 \in \cK_2} \IP[B_{\ov{m}_1}\cap B_{\ov{m}_2}] \,,
\end{equation}
with $c_0 \ge 1$, a suitable constant.

\medskip
Given $\ov{m}_1$ in $\cK_1$ and $\ov{m}_2$ in $\cK_2$ we now write
\begin{equation}\label{2.19}
V = \wt{D}_{\ov{m}_1} \cup \wt{D}_{\ov{m}_2} \subseteq \IZ^d\,,
\end{equation}

\n
and introduce the decomposition of $\mu_{V,u}$, (see (\ref{1.34}) for the notation),
\begin{align}
\mu_{V,u}  = & \delta_{1,1} + \delta_{1,2} + \delta_{2,1} + \delta_{2,2}, \;\mbox{where for $i,j$ distinct in $\{1,2\}$:}\label{2.20}
\\[1ex]
\delta_{i,j} & = 1\{X_0 \in \wt{D}_{\ov{m}_i}, \;H_{\wt{D}_{\ov{m}_j}} < \infty\} \;\mu_{V,u} \;\;\mbox{and} \nonumber
\\
\delta_{i,i} & = 1\{X_0 \in \wt{D}_{\ov{m}_i}, \;H_{\wt{D}_{\ov{m}_j}} = \infty\} \;\mu_{V,u} \,.\nonumber
\end{align}
Observe that thanks to (\ref{1.35}),
\begin{equation}\label{2.21}
\mbox{the $\delta_{i,j}$, $1 \le i,j \le 2$, are independent Poisson point processes on $W_+$},
\end{equation}

\n
with respective intensity measures $\zeta_{i,j}$, $1 \le i, j \le 2$, where for $1 \le i \not= j \le 2$, we set
\begin{equation}\label{2.22}
\begin{split}
\zeta_{i,j}(dw) & = u \,1\{X_0 \in \wt{D}_{\ov{m}_i}, \;H_{\wt{D}_{\ov{m}_j}} < \infty\} \,P_{e_V}(dw)\,,
\\
\zeta_{i,i}(dw) & = u \,1\{X_0 \in \wt{D}_{\ov{m}_i}, \;H_{\wt{D}_{\ov{m}_j}} = \infty\} \,P_{e_V}(dw)\,.
\end{split}
\end{equation}

\n
Given a random variable $\Lambda(\omega)$ on $(\Omega, \cA)$ with values in the set of finite point measures on  $W_+$, for $\ov{m} \in J_n$ we define:
\begin{equation}\label{2.23}
\begin{split}
B_{\ov{m}}(\Lambda)  = \{ & \o \in \Omega; \;\mbox{there is a $*$-path in $\IZ^2$ from $\wt{\cD}_{\ov{m}}$ to $\wt{\cV}_{\ov{m}}$, such that}
\\
& \mbox{for each $y$  in the path $\o \notin \cG^\Lambda_y\}$}\,,
\end{split}
\end{equation}

\medskip\n
where $\cG^\Lambda_y$ is defined analogously as in (\ref{2.3}), with $\big(\bigcup_{w \in {\rm Supp}\,\Lambda (\o)} w(\IN)\big)^c$ replacing $\cV^u(\o)$. Note that $B_{\ov{m}}(\Lambda)$ naturally increases with $\Lambda$.

\medskip
Due to the definition of $\wt{\cV}_{\ov{m}}$ in (\ref{2.13}), the event $B_{\ov{m}}$ only depends on the restriction $\cV^u(\o) \cap \wt{D}_{\ov{m}}$ of the vacant set at level $u$ to $\wt{D}_{\ov{m}}$. Thus with (\ref{1.38}), $B_{\ov{m}_i} = B_{\ov{m}_i} (\mu_{V,u})$, for $i = 1,2$. In addition due to the fact that $(\bigcup_{w \in {\rm Supp} \,\delta_{i,i}} w(\IN) \cap \wt{D}_{\ov{m}_j}) = \emptyset$, for $1 \le i \not= j \le 2$, we find that
\begin{equation}\label{2.24}
\begin{split}
B_{\ov{m}_1} = B_{\ov{m}_1} (\mu_{V,u}) & = B_{\ov{m}_1} (\delta_{1,1} + \delta_{1,2} + \delta_{2,1})\,,
\\
B_{\ov{m}_2} = B_{\ov{m}_2} (\mu_{V,u}) & = B_{\ov{m}_2} (\delta_{2,2} + \delta_{2,1} + \delta_{1,2})\,.
\end{split}
\end{equation}
It then follows that
\begin{equation}\label{2.25}
\begin{array}{l}
\IP [B_{\ov{m}_1} \cap B_{\ov{m}_2}] \le \IP [B_{\ov{m}_1} (\delta_{1,1}) \cap B_{\ov{m}_2} (\delta_{2,2}), \,\delta_{1,2} = \delta_{2,1} = 0] \;+
\\
\IP[\delta_{1,2}  \;\mbox{or} \; \delta_{2,1} \not= 0] \stackrel{(\ref{2.21})}{\le} \IP [B_{\ov{m}_1} (\delta_{1,1}) ] \,\IP[B_{\ov{m}_2} (\delta_{2,2})] + \IP[\delta_{1,2} \not= 0] + \IP[\delta_{2,1} \not= 0]
\\[1ex]
\le \IP[B_{\ov{m}_1} (\mu_{V,u})] \,\IP[B_{\ov{m}_2} (\mu_{V,u})] + \zeta_{1,2}  (W_+) + \zeta_{2,1}(W_+)
\\[1ex]
\le q_n^2 + \ve_n, \;\mbox{where in the notation of (\ref{1.33})}
\end{array}
\end{equation}
\begin{equation}\label{2.26}
\ve_n = 2 \,u \, \sup\limits_{\ov{m}_1 \in \cK_1, \ov{m}_2 \in \cK_2}\, \cE(\wt{D}_{\ov{m}_1}, \wt{D}_{\ov{m}_2})\,,
\end{equation}

\medskip\n
and we used a similar calculation as in (\ref{1.33}) to bound $\zeta_{1,2}(W_+)$ and $\zeta_{2,1}(W_+)$.

\medskip
Coming back to (\ref{2.18}) we thus find that for $n \ge 0$,
\begin{equation}\label{2.27}
q_{n+1} \le c_0 \,\ell^2_n (q^2_n + \ve_n)\,.
\end{equation}
Setting for $n \ge 0$,
\begin{equation}\label{2.28}
b_n = c_0 \,\ell_n^2 \,q_n, \;\mbox{(where we recall that $c_0 \ge 1$, see (\ref{2.18}))},
\end{equation}
we find that for $n \ge 0$,
\begin{equation}\label{2.29}
b_{n+1} = c_0\,\ell^2_{n+1} \, q_{n+1} \le c_0^2 \,\ell_{n+1}^2 \,\ell^2_n (q_n^2 + \ve_n) = \Big(\dis\frac{\ell_{n+1}}{\ell_n}\Big)^2 \,b_n^2 + c(\ell_{n+1} \,\ell_n)^2 \ve_n\,.
\end{equation}

\medskip\n
The next lemma explains how we will use the above recurrence relation.

\begin{lemma}\label{lem2.3} (under {\rm (\ref{2.6})})
\begin{equation}\label{2.30}
\begin{array}{l}
\mbox{For $L_0 \ge c$, when for all $n \ge 0, \ve_n \le L_n^{-1}$, then $b_0 \le L_0^{-\frac{1}{2}}$} 
\\
\mbox{implies that $b_n \le L_n^{-\frac{1}{2}}$, for all $n \ge 0$}\,.
\end{array}
\end{equation}
\begin{equation}\label{2.31}
\hspace{-3ex}\mbox{For $d \ge c$, $L_0 \ge c_1\,d$, one has $\ve_n \le L_n^{-1}$, for all $n \ge 0$}\,.
\end{equation}
\end{lemma}

\begin{proof}
We begin with the proof of (\ref{2.30}). We first note that for $n \ge 0$,
\begin{equation}\label{2.32}
\begin{array}{rl}
{\rm i)} &\quad  \dis\frac{\ell_{n+1}}{\ell_n} \stackrel{(\ref{2.9})}{\le} 2 \;\dis\frac{L^a_{n+1}}{L^a_n} \le 2(100 L_n^{1+a})^a \;L^{-a}_n \stackrel{(\ref{2.8})}{\le} 200 L_n^{a^2}\,,
\\[2ex]
{\rm ii)} &\quad  (\ell_{n+1} \,\ell_n)^2 \le c \, L_{n+1}^{2a} \;L^{2a}_n \le c \,L_n^{4a + 2a^2}\,.
\end{array}
\end{equation}

\medskip\n
As a result when $b_n \le L_n^{-\frac{1}{2}}$ holds we find that:
\begin{equation}\label{2.33}
\begin{array}{lcl}
b_{n+1} & \hspace{-1ex} \stackrel{(\ref{2.29}), (\ref{2.32})}{\le} & \hspace{-3ex} c \,L_n^{4a + 2a^2} (L_n^{-1} + L_n^{-1}) = c\,L^{-\frac{1}{2}}_{n+1} \;L_{n+1}^{\frac{1}{2}} \;L_n^{4a + 2a^2 -1}
\\[1ex]
& \hspace{-1ex} \le & \hspace{-3ex} L_{n+1}^{-\frac{1}{2}} \;c^\prime \,L_n^{\frac{a}{2} + \frac{1}{2} + 4a + 2a^2 - 1} \stackrel{(\ref{2.8})}{\le} L_{n+1}^{-\frac{1}{2}}\;c^\prime \,L_{n}^{-\frac{1}{4}} \le L_{n+1}^{-\frac{1}{2}}, \;\mbox{when $L_0 \ge c$}\,.
\end{array}
\end{equation}

\medskip\n
The claim (\ref{2.30}) follows. We then turn to the proof of (\ref{2.31}). We find with (\ref{2.26}), and (\ref{1.33}), (\ref{1.24}), that for $n \ge 0$,
\begin{equation}\label{2.34}
\ve_n \le 2u \,|\wt{D}_{\ov{m}}|^2 \;\ \sup\limits_{\ov{m}_i \in \cK_i, i=1,2} \;\sup\limits_{x_i \in \wt{D}_{\ov{m}_i}, i=1,2} g(x_1,x_2), \;\mbox{with $\ov{m} \in J_n$ arbitrary}\,.
\end{equation}

\n
Observe that for $\ov{m}$, $x_1,x_2$ as above we have
\begin{equation*}
|\wt{D}_{\ov{m}}| \le 9 \,L^2_n \,2^d, \;\mbox{and} \;|x_1 - x_2|_1 \ge \frvier\; L_{n+1} \,.
\end{equation*}

\n
Thus with (\ref{2.6}) we know that $u \le d ( \le 2^d)$, and when $d \ge 5$, due to (\ref{1.11}), we have
\begin{equation}\label{2.35}
\ve_n \le c\,d \,L_n^4 \,2^{2d} (c^\prime \, d / L_{n+1})^{\frac{d}{2} -2} \le c\,L^4_n \,(\wt{c}\,d / L_{n+1})^{(\frac{d}{2} -2)} \le c\,L^4_n \,\ell_n^{-(\frac{d}{2} - 2)} ,
\end{equation}

\n
when $L_0 \ge \wt{c} \,d$. As a result we find with (\ref{2.9}) that
\begin{equation}\label{2.36}
\ve_n \le c \,L_n^4 \,L_n^{-6} \le L_n^{-1}\,,
\end{equation}

\n
when $(\frac{d}{2} - 2) \ge \frac{6}{a} \; (\stackrel{(\ref{2.8})}{=} 600)$, and $L_0 \ge c \vee (\wt{c} \,d)$. The claim (\ref{2.31}) follows.
\end{proof}

We will now rely on (\ref{2.4}) to establish that for large $d$,
\begin{equation}\label{2.37}
b_0 \le L_0^{-\frac{1}{2}}, \;\mbox{when} \;L_0 = [c_1\,d] + 1\,.
\end{equation}

\n
With the above lemma it will then follow that for large $d$, with this choice of $L_0$, one has
\begin{equation}\label{2.38}
\ell^2_n \,q_n \stackrel{(\ref{2.28})}{\le} b_n \le L_n^{-\frac{1}{2}}, \;\mbox{for all $n \ge 0$} \,.
\end{equation}
To prove (\ref{2.37}) we note that
\begin{equation}\label{2.39}
b_0 = c_0 (100 [L_0^a])^2 \,q_0 \le c\,L_0^{2a} \,\IP \Big[\bigcup\limits_{y \in [-L_0,2L_0)^2} \,\cG^c_{y,u}\Big] \le c\,L_0^{2a + 2} \IP[\cG_u^c]\,.
\end{equation}

\n
As a result we see that when $d$ tends to infinity
\begin{equation*}
L_0^{\frac{1}{2}} \;b_0 \le c\,L_0^{2a + 2,5} \,\IP[\cG_u^c] \stackrel{(\ref{2.4})}{\longrightarrow} 0\,,
\end{equation*}

\n
and the claim (\ref{2.37}) follows. Thus (\ref{2.38}) holds for large $d$. For such $d$ and $L_0$ as in (\ref{2.37}), we have
\begin{equation*}
\begin{array}{l}
\mbox{$\IP[0$ belongs to a finite connected component of $\{y \in \IZ^2; \;\o \in \cG_{y,u}\}] \le$}
\\[1ex]
 \IP [\bigcup\limits_{|y|_\infty < L_0} \cG^c_{y,u}] + \IP [\mbox{there is a $*$-circuit in $\IZ^2 \backslash (-L_0,L_0)^2$ containing $0$ in the interior,}
\\[2ex]
\mbox{such that for each $y$ in the circuit $\cG^c_{y,u}$ occurs$] \le cL_0^2 \,\IP[\cG^c_u] \; +$}
\\[1ex]
\mbox{$\dsl_{n \ge 0} \,\IP[$there is a $*$-circuit in $\IZ^2 \backslash (-L_0 \,L_0)^2$ containing $0$ in  the interior,}
\\[2ex]
\mbox{passing through $[L_n\,L_{n+1}) e_1$, such that $\cG^c_{y,u}$ occurs for each $y$ in the circuit$]$.}
\end{array}
\end{equation*}

\medskip\n
Aa a result we see that 
\begin{equation}\label{2.40}
\begin{array}{l}
\IP[0  \;\mbox{belongs to a finite connected component of}\; \{y \in \IZ^2; \,\omega \in \cG_{y,u}\}] \le
\\[1ex]
c\,d^2 \,\IP[\cG^c_u] \;+ \dsl_{n \ge 0} \;\dsl_{m \in J_n \atop \cD_m \cap [L_n, L_{n+1})e_1 \not= \emptyset} 
\IP [B_m] \stackrel{(\ref{2.38})}{\le} c\,d^2 \,\IP[\cG^c_u]
 + \dsl_{n \ge 0} \ell_n \,L_n^{-\frac{1}{2}} \,\ell_n^{-2} \stackrel{(\ref{2.9})}{\le}
\\[1ex]
c\,d^2 \,\IP[\cG^c_u] + \dsl_{n \ge 0} (c_1\,d)^{-\frac{1}{2} (1 + a)^n} \stackrel{(\ref{2.4})}{<} \frac{1}{2}, \;\mbox{for sufficiently large $d$} \,.
\end{array}
\end{equation}

\n
In particular this shows (\ref{2.7}), and thus concludes the proof of Theorem \ref{theo2.2}.
\end{proof}

\begin{remark}\label{rem2.4} \rm 
If $t_x, x \in \IZ^d$, stands for the canonical shift on $\{0,1\}^{\IZ^d}$, it follows from (2.6) of \cite{Szni07a}, that restricting $x$ to $\IZ^2$, one obtains an ergodic measure preserving flow on $(\{0,1\}^{\IZ^d}, \cY, Q_u)$, with $\cY$ the canonical $\sigma$-algebra on $\{0,1\}^{\IZ^d}$. Thus (\ref{2.7}) not only shows that for large $d$, with full $\IP$-probability $\cV^{u(d)}$ percolates but even that
\begin{equation}\label{2.41}
\begin{array}{l}
\mbox{when (\ref{2.4}) holds, for $d \ge d_0$, $\IP$-a.s., $\cV^{u(d)} \cap (\IZ^2 + C)$} 
\\
\mbox{contains an infinite connected component.}
\end{array}
\end{equation}
Thus for large $d$, $\cV^{u(d)}$ percolates in $\IZ^2 + C$. \hfill $\square$
\end{remark}

\section{Growing trees}
\setcounter{equation}{0}

The objective of the present and the next sections is to show that (\ref{2.4}) holds for sequences $u(d)$ equivalent to $(1-3\ve) \log d$, with $\ve$ an arbitrarily small positive number. The main Theorem \ref{theo0.1} then follows from Theorem \ref{theo2.2}. In this section, as a step towards this goal, we introduce $u_0 \sim (1-\ve) \log d$, see (\ref{3.2}) below, and show that for large $d$, with high probability, most vertices in $C$ have some neighbor belonging to a substantial connected component of $\cV^{u_0} \cap C$, containing at least $L \ge c(\ve) \,d^{c^\prime \log \frac{1}{\ve}}$ vertices, with $c^\prime$ a large constant, see (\ref{3.5}). The main result of this section appears in Theorem \ref{theo3.1}. Its proof involves the construction of certain large trees in $\cV^{u_0} \cap C$. These trees are obtained by trimming larger trees which show up as part of a ``first approximation'' of the vacant set, when $d$ is large.

\medskip
Throughout the remainder of this section and the next section, constants will implicitly depend on the parameter
\begin{equation}\label{3.1}
\ve \in \Big(0, \mbox{\f $\dis\frac{1}{10}$}\Big)\,.
\end{equation}

\n
We then define the level $u_0$ via, see (\ref{1.6}) for the notation,
\begin{equation}\label{3.2}
u_0 = (1-\ve) \,g(0) \log d\,,
\end{equation}
(we recall that $g(0)$ tends to $1$ as $d$ goes to infinity, cf.~(\ref{1.9})), as well as the constant
\begin{equation}\label{3.3}
\alpha = 15 + e + \log 500 / \log 3\,,
\end{equation}

\n
which is such that when $N$ is a Poisson variable with parameter $\frac{u_0}{g(0)}$, (think of the number of trajectories modulo time-shift with label at most $u_0$ going through a given site $x$ in $\IZ^d$, cf.~(\ref{1.26}), (\ref{1.34}), (\ref{1.35})), one has
\begin{equation}\label{3.4}
P[N \ge \alpha \, \log d] \le \exp\{ - \alpha \,\log d + (1-\ve) \log d(e-1)\} \le d^{-14} / 500\,.
\end{equation}

\n
We aim at constructing trees of depth $\ell -1$ in $\cV^{u_0} \cap C$ which contain at least $L$ sites, where $\ell$ and $L$ are determined through:
\begin{equation}\label{3.5}
\ell = 3 + \Big[\ve^{-1} (7 + 2 \alpha \,\log \,\mbox{\f $\dis\frac{1}{\ve}$}\Big)\Big], \;\mbox{and} \; L = \Big[ \fr(d^\ve / \ell)^{\ell - 1}\Big]\,.
\end{equation}
In what follows the expression {\it substantial component} of $C$ will stand for a connected component of $\cV^{u_0} \cap C$ with at least $L$ points. Our main result is
\begin{theorem}\label{theo3.1}
\begin{equation}\label{3.6}
\begin{array}{l}
\mbox{$\IP[$at most $|C|\,e^{-c_2 d^{\ve/2}}$ points of $C$ have no neighbor in a}
\\
\mbox{substantial component$] \ge 1-c\,e^{-cd^{\ve/2}}$}.
\end{array}
\end{equation}
\end{theorem}

\begin{proof}
With translation invariance, symmetry, and Chebyshev's inequality, it suffices to prove that:
\begin{equation}\label{3.7}
\mbox{$\IP[$no neighbor of $0$ belongs to a substantial component$] \le c \,e^{-cd^{\ve/2}}$}\,.
\end{equation}

\n
From now on, without loss of generality we assume $d > \ell$, and consider $I_1,\dots , I_\ell$, the pairwise disjoint consecutive intervals of $\{1,\dots , d\}$ with length $[\frac{d}{\ell}]$. We introduce a deterministic tree rooted at the origin through the formula
\begin{equation}\label{3.8}
\begin{split}
\cT =\{ & x \in B_1(0,\ell); \;x = e_{i_1} + \dots + e_{i_k}, \;\mbox{for some} 
\\
& 0 \le k \le \ell, \;i_1 \in I_1,\dots, i_k \in I_k\} \subset C .
\end{split}
\end{equation}

\n
For $0 \le k \le \ell$, the $k$-th generation of $\cT$ is
\begin{equation}\label{3.9}
\cT_k = \cT \cap S_1 (0,k) \;\mbox{and contains $[\frac{d}{\ell}]^k$ points}\,.
\end{equation}

\n
We will now construct a random subtree $\cT^0$ of $\cT$, which is approximately a supercritical Galton-Watson tree of depth $\ell$. As we will see in Lemma \ref{lem3.3}, for large $d$, $\cT^0 \backslash \{0\}$ essentially lies in $\cV^{u_0} \cap C$. In order to define this random tree we introduce for $x$ in $\IZ^d$, the set $W^*_x \subseteq W^*$, (not to be confused with $W^*_{\{x\}}$, see~(\ref{1.30})), via
\begin{equation}\label{3.10}
W^*_x = \pi^* (\{w \in W; \;\mbox{the first $w(n), n \in \IZ$, with minimal $|\cdot |_1$-norm is $x$\})}\,,
\end{equation}
as well as the random variables
\begin{equation}\label{3.11}
n_x(\omega) = \omega(W^*_x \times [0,u_0])\,.
\end{equation}

\n
Analogously defined objects are very useful in the discussion of the percolative properties of the vacant set of random interlacements on trees, see Section 5 of Teixeira \cite{Teix08b}.

\medskip
The sets $W^*_x$, $x \in \IZ^d$, constitute a measurable partition of $W^*$, and therefore
\begin{equation}\label{3.12}
\begin{array}{l}
\mbox{$n_x, x \in \IZ^d$, are independent Poisson variables with}
\\
\mbox{respective parameters $u_0 \,\nu(W_x^*)$}\,.
\end{array}
\end{equation}

\n
In view of (\ref{1.31}), (\ref{1.32}) and (\ref{1.24}), we have the identity
\begin{equation}\label{3.13}
\nu(W_x^*) = P_x[|X_n|_1 > |x|_1, \;\mbox{for all $n \ge 1] \;P_x [|X_n|_1 \ge |x|_1$ for all $n \ge 0]$}\,.
\end{equation}

\n
As a result of (\ref{1.12}), we also find that
\begin{equation}\label{3.14}
1 -\mbox{\f $\dis\frac{c}{d}$} \le \nu(W^*_x) \le 1, \;\mbox{for $x \in B_1(0,\ell)$}\,.
\end{equation}

\n
We can now introduce the above mentioned random subtree of $\cT$, as the component of $0$ in $\cT$ of the sites $x$ with vanishing occupation number $n_x$:
\begin{equation}\label{3.15}
\begin{split}
\cT^0 = \{0\} \cup \{& x \in \cT \backslash \{0\}; \; n_{x_1} = \dots = n_{x_k} = 0, \;\mbox{for $x_0 = 0, \dots , x_k = x$},
\\
&\mbox{the canonical path from $0$ to $x$ in $\cT\}$}\,.
\end{split}
\end{equation}

\n
We will now see that $\cT^0$ is typically large, and for this purpose we define the variables, (which in view of (\ref{3.12}) are in fact independent):
\begin{equation}\label{3.16}
D(z) = \dsl_{i \in I_{|z|_1 + 1}} 1\{n_{z + e_i} = 0\}, \;\mbox{for $z \in \cT \cap B_1(0,\ell-1)$} \,.
\end{equation}

\n
Due to (\ref{3.12}), (\ref{3.14}), when $d$ is large the distribution of $D(z)$ is approximately binomial with parameters $n = [\frac{d}{\ell}]$ and $p = \exp\{-u_0 (1 - O(\frac{1}{d}))\} \stackrel{(\ref{1.9}),(\ref{3.2})}{=} d^{-(1-\ve)}(1 + O(\frac{\log d}{d}))$. The next lemma contains the precise estimates that we will need.

\begin{lemma}\label{lem3.2}
\begin{equation}\label{3.17}
\begin{array}{l}
\IP[\cD] \ge 1 - c \,e^{-cd^{\ve}}, \;\mbox{where}
\\
\cD = \Big\{\mbox{for all} \; z \in \cT \cap B_1(0,\ell-1), \;e^{-\frac{10^{-3}}{\ell}} \;\dis\frac{d^\ve}{\ell} \le D(z) \le e^{\frac{10^{-3}}{\ell}} \; \dis\frac{d^\ve}{\ell}\Big\}\,.
\end{array}
\end{equation}
\end{lemma}

\begin{proof}
For $z$ as above, $D(z)$ is distributed, thanks to (\ref{3.12}), as the sum of $[\frac{d}{\ell}]$ independent Bernoulli variables with parameters $e^{-u_0 \nu(W^*_{z + e_i})}$, $i \in I_{|z|_1 + 1}$. With (\ref{3.14}), (\ref{1.9}) and an exponential Chebychev bound we see that for $\lambda > 0$,
\begin{equation}\label{3.18}
\begin{array}{l}
\IP\Big[D(z) > e^{\frac{10^{-3}}{\ell}} \;\mbox{\f $\dis\frac{d^\ve}{\ell}$}\Big] \le 
\\
\exp\Big\{ - \lambda \,e^{\frac{10^{-3}}{\ell}} \;\mbox{\f $\dis\frac{d^\ve}{\ell}$} + \Big[\mbox{\f $\dis\frac{d}{\ell}$}\Big] \,\log \Big(1 + (e^\lambda - 1) \dis\frac{(1+c(\log d)/ d)}{d^{1-\ve}}\Big)\Big\}\,.
\end{array}
\end{equation}

\n
Picking $\lambda = c$ small, we find that the left-hand side of (\ref{3.18}) is smaller than $\exp\{ -c \,d^\ve\}$, when $d \ge c$. In a similar fashion we see that for $\lambda > 0$,
\begin{equation}\label{3.19}
\begin{array}{l}
\IP\Big[D(z) < e^{-\frac{10^{-3}}{\ell}} \;\mbox{\f $\dis\frac{d^\ve}{\ell}$}\Big] \le
\\
\exp\Big\{  \lambda \,e^{-\frac{10^{-3}}{\ell}} \;\mbox{\f $\dis\frac{d^\ve}{\ell}$} + \Big[\mbox{\f $\dis\frac{d}{\ell}$}\Big] \,\log \Big(1 - (1- e^{-\lambda}) \dis\frac{(1-c(\log d)/ d)}{d^{1-\ve}}\Big)\Big\}\,,
\end{array}
\end{equation}

\n
and choosing $\lambda = c$ small, we see that the left-hand side of (\ref{3.19}) is at most $\exp\{-c\,d^\ve\}$, for $d \ge c$. Since $|\cT| \le |B_1(0,\ell)| \le (2d)^{\ell + 1}$, the claim (\ref{3.17}) readily follows.
\end{proof}

We thus see that on the typical event $\cD$, the random tree $\cT^0$ is sizeable. However it may well intersect $\cI^{u_0}$. For this reason we will needs to prune $\cT^0$ in order to construct a substantial component of $\cV^{u_0} \cap C$ abutting $0$. With this in mind, we consider the various possible ``shapes'' of $\cT^0$ compatible with $\cD$, i.e. the collection $\cS$ of finite subtrees $\cA$ of $\cT$, rooted at $0$, with depth $\ell$ such that
\begin{equation}\label{3.20}
\{\cT^0 = \cA\} \cap \cD \not= \emptyset\,.
\end{equation}

\n
With a similar notation as in (\ref{3.9}) we write $\cT^0_k = \cT^0 \cap S_1(0,k)$ and  $\cA_k = \cA \cap S_1 (0,k)$, for $0 \le k \le \ell$ and $\cA$ in $\cS$. Note that for $\cA$ in $\cS$ we have
\begin{equation}\label{3.21}
\begin{array}{l}
\{\cT^0 = \cA\} = A_1 \cap \dots \cap A_\ell, \;\mbox{where for $1 \le k \le \ell$}\,,
\\
A_k = \{\omega \in \Omega; \;n_x = 0$, for $x \in \cA_k$, and $n_x > 0$, for $x \in (\cT_k \backslash \cA_k) \cap \partial \cA_{k-1}\} \,.
\end{array}
\end{equation}

\n
We also write $A_0 = \Omega$ by convention. Due to (\ref{3.12}) we see that 
\begin{equation}\label{3.22}
A_i, 1 \le i \le \ell, \;\mbox{are independent} \,.
\end{equation}

\n
Our aim is to show that $\cI^{u_0}$ has a thin trace on $\cT^0$, when $d$ is large. As a preparation we make the following observation. Consider $\cA \in \cS$, $\omega \in \{\cT^0 = \cA\}$, $0 \le k \le \ell$, and $x \in \cA_k$. We know from (\ref{3.10}) and (\ref{1.30}) that $W^*_z \cap W^*_{\{x\}} = \emptyset$, when $|z|_1 > k$ $( = |x|_1$). On the other hand the events $W^*_z$, $z \in \IZ^d$, partition $W^*$. Therefore, if $x \in \cI^{u_0}(\omega)$, then $\omega((W_z^* \cap W^*_{\{x\}}) \times [0,u_0]) \ge 1$, for some $z$ in $B_1(0,k)$. This implies that on $\{\cT^0 = \cA\}$,
\begin{align}
& |\cA_k \cap \cI^{u_0}| \le \dsl_{0 \le m \le k} N_{m,k}, \;\mbox{where}\label{3.23}
\\[1ex]
&N_{m,k} = \dsl_{|z|_1 = m, x \in \cA_k} \o((W_z^* \cap W^*_{\{x\}}) \times [0,u_0]), \;\;\mbox{for $0 \le m \le k \le \ell$} \,. \label{3.24}
\end{align}

\n
The crucial control on the trace of $\cI^{u_0}$ on $\cT^0$ comes in the next lemma.

\begin{lemma}\label{lem3.3}
\begin{equation}\label{3.25}
\IP[| \cT^0_k \cap \cI^{u_0}| \ge d^{-\ve/2} | \cT^0_k|] \le c\,\exp\{ - c\,d^{\ve/2}\}, \;\;\mbox{for $1 \le k \le \ell$} \,.
\end{equation}
\end{lemma}

\begin{proof}
In view of (\ref{3.17}), of the fact that $\cD \subseteq \bigcup_{\cA \in \cS} \{ \cT^0 = \cA\}$, (this is in fact an equality), and (\ref{3.23}), it suffices to show that for $1 \le k \le \ell$, $0 \le m \le k$,
\begin{equation}\label{3.26}
\sup\limits_{\cA \in \cS} \IP[N_{m,k} \ge (\ell + 1)^{-1} \,d^{-\ve/2} |\cA_k| \;|\cT^0 = \cA] \le c\, \exp\{-cd^{\ve/2}\}\,.
\end{equation}

\n
The variables $n_z$, with $|z|_1 \not= m$, are independent of $N_{m,k}$. It now follows from (\ref{3.21}) that for $\lambda > 0$, $0 \le m \le k \le \ell$, $\cA \in \cS$, we have:
\begin{equation}\label{3.27}
\begin{array}{l}
\IP[N_{m,k} \ge (\ell + 1)^{-1} \,d^{-\ve/2} |\cA_k|\; | \cT^0 = \cA ] = \IP [N_{m,k} \ge (\ell + 1)^{-1} \,d^{-\ve/2} |\cA_k | \; | A_m] \le \;
\\[1ex]
\exp\{-\lambda (\ell + 1)^{-1} \,d^{-\ve/2} \,|\cA_k| \} \,\IE [\exp\{ \lambda \,N_{m,k}\} | A_m]\,.
\end{array}
\end{equation}

\n
We will bound the last term of (\ref{3.27}). In the end we will pick $\lambda = 1$, see (\ref{3.49}), but it will be instructive to keep $\lambda$ generic in the mean time. Part of the difficulty will have to do with the conditioning present in the last term of (\ref{3.27}). We will first prove (\ref{3.39}), and afterwards bound the right-hand side of (\ref{3.39}), see (\ref{3.46})~-~(\ref{3.48}). We introduce the notation
\begin{equation}\label{3.28}
\begin{array}{l}
f_k(w^*) = \dsl_{n \in \IZ} 1\{X_n (w) \in \cA_k\}, \;\mbox{for $w^* \in W^*$, with $w \in W$ arbitrary}
\\
\mbox{such that $\pi^*(w) = w^*$} ,
\end{array}
\end{equation}

\n
The next step is to derive the upper bound, which appears in (\ref{3.39}) below. We denote with $\langle \o, h\rangle = \sum_i h(w_i^*,u_i)$, for $\o = \sum_{i \ge 1} \delta_{(w^*_i,u_i)}$, the integral of a non-negative measurable $h$ on $W^* \times \IR_+$, with respect to the point measure $\omega$. As a result of (\ref{3.24}) we see that for $0 \le m \le k \le \ell$,
\begin{equation}\label{3.29}
N_{m,k} \le \dsl_{|z|_1 = m} \big\langle \o, \;(1_{W^*_z} f_k) \otimes 1_{[0,u_0]}\big\rangle \,.
\end{equation}
Since the events $W^*_z, z \in \IZ^d$, partition $W^*$, the random vectors
\begin{equation}\label{3.30}
\big(n_z, \big\langle \o, (1_{W^*_z} f_k) \otimes 1_{[0, u_0]}\big\rangle\big), \; z \in S_1(0,m), \;\mbox{are independent}\,.
\end{equation}

\n
In view of the definition of $A_m$, see (\ref{3.21}), we thus see that
\begin{align}
&\IE [\exp\{\lambda\,N_{m,k}\} \,| A_m] \le a_1\,a_2, \;\mbox{where} \label{3.31}
\\[2ex]
&a_1 = \mbox{\f $\dis\prod\limits_{z \in \cZ_1}$} \IE[\exp\{\lambda \langle \omega, (1_{W^*_z} f_k) \otimes 1_{[0, u_0]}\rangle \} ]  \label{3.32}
\\
&a_2 = \mbox{\f $\dis\prod\limits_{z \in \cZ_2}$} \IE[\exp\{\lambda \langle \omega, (1_{W^*_z} f_k) \otimes 1_{[0, u_0]}\rangle \} | n_z >0 ] \nonumber
\intertext{and we have set}
& \cZ_1 = S_1(0,m) \backslash (\cA_m \cup \partial \cA_{m-1}), \;\mbox{when $1 \le m \le \ell$, and $\cZ_1 = \{0\}$, when $m = 0$}, \nonumber
\\
& \cZ_2 = (S_1(0,m) \backslash \cA_m) \cap \partial \cA_{m-1}, \;\mbox{when $1 \le m \le \ell$, and $\cZ_2 = \emptyset$, when $m = 0$}. \nonumber
\end{align}

\n
From the definition of $\IP$, see below (\ref{1.30}), we find that
\begin{equation}\label{3.33}
a_1 = \exp\Big\{\dsl_{z \in \cZ_1} u_0 \,\dis\int_{W_z^*} (e^{\lambda f_k} - 1) \, d\nu\Big\}\,.
\end{equation}

\medskip\n
We will now bound $a_2$. For this purpose we consider $z \in \cZ_2$ and note that
\begin{equation}\label{3.34}
\begin{array}{l}
n_z = X+Y, \;\mbox{where} \; X= \o((W^*_z \cap W^*_{\cA_k}) \times [0,u_0]) \;\mbox{and} 
\\[0.5ex]
Y = \o ((W^*_z \backslash W^*_{\cA_k}) \times [0, u_0]) \;\mbox{are independent Poisson}
\\[0.5ex]
\mbox{variables with respective intensity $u_0 \,\nu(W^*_z \cap W^*_{\cA_k})$} 
\\[0.5ex]
\mbox{and} \; u_0 \,\nu(W^*_z \backslash W^*_{\cA_k}) \,.
\end{array}
\end{equation}

\n
We also introduce the variable
\begin{equation}\label{3.35}
\begin{array}{l}
Z = \exp\big\{ \lambda \big\langle \o, (1_{W^*_z} f_k) \otimes 1_{[0,u_0]} \big\rangle\big\}, \;\mbox{which is independent of $Y$},
\\[0.5ex]
\mbox{(indeed $f_k$ vanishes on $W^*_z \backslash W^*_{\cA_k}$)}.
\end{array}
\end{equation}

\n
The generic factor in the product defining $a_2$ in (\ref{3.32}) equals
\begin{equation}\label{3.36}
\begin{array}{l}
\IE[Z | X + Y \ge 1] = \IE[Z | Y \ge 1] \,\IP[Y \ge 1 | X+Y \ge 1] \;+ 
\\[1ex]
\IE[Z |Y = 0, X \ge 1] \,\IP[Y = 0 | X+Y \ge 1] \stackrel{\rm independence}{=}
\\[2ex]
\IE[Z] ( 1- \IP[Y = 0 | X+Y \ge 1]) + \IE[Z | X \ge 1] \,\IP[Y=0 | X + Y \ge 1]\,.
\end{array}
\end{equation}

\n
Since $Z = 1$ on $\{X=0\}$, we also have $\IE[Z | X \ge 1] = (\IE[Z] - \IP[X=0]) / \IP[X \ge 1]$. Inserting this identity in the last line of (\ref{3.36}) we find that
\begin{equation}\label{3.37}
\begin{array}{l}
\IE[Z | X + Y \ge 1] = \IE[Z] \Big(1 + \dis\frac{\IP[X=0]}{\IP[X \ge 1]} \;\Big(1 - \dis\frac{1}{\IE[Z]}\Big) \,\IP[Y = 0 | X + Y \ge 1]\Big) =
\\[3ex]
\IE[Z]  \Big(1 + \IP[X=0] \Big(1 - \dis\frac{1}{\IE[Z]}\Big)\;\dis\frac{\IP[Y=0]}{\IP[n_z > 0]}\Big) =
\\[3ex]
 \IE[Z]\Big(1 + \Big(1-e^{-u_0  \int_{W^*_z} (e^{\lambda \,f_k} -1)d \nu}\Big) \;\dis\frac{e^{-u_0 \nu(W^*_z)}}{1-e^{-u_0 \nu(W^*_z)}}\Big)\,.
\end{array}
\end{equation}

\n
Using the inequality $1-e^{-u} \le u$, we find that
\begin{equation}\label{3.38}
\begin{array}{l}
\IE[Z | X+ Y \ge 1] \le \IE[Z] \,\exp\Big\{\dis\frac{e^{-u_0 \nu(W^*_z)}}{1-e^{-u_0 \nu(W^*_z)}} \;u_0 \dis\int_{W^*_z} (e^{\lambda f_k} -1) d \nu\Big\} =
\\[2ex]
\exp\Big\{\dis\frac{1}{1-e^{-u_0 \nu(W^*_z)}}  \;u_0 \dis\int_{W^*_z} (e^{\lambda f_k} -1) d \nu\Big\}\,.
\end{array}
\end{equation}

\n
With (\ref{3.14}), (\ref{3.2}), we know that $e^{-u_0 \nu(W^*_z)} \le c\,d^{-(1-\ve)}$. Coming back to (\ref{3.31}) and noting that $\cZ_1 \cup \cZ_2$ equals $S_1(0,m) \backslash \cA_m$, when $1 \le m \le \ell$, and $\{0\}$, when $m = 0$, we see that for $d \ge c$, and $0 \le m \le k$, $1 \le k \le \ell$, we have
\begin{equation}\label{3.39}
\IE[\exp\{ \lambda\,N_{m,k}\} | A_m] \le \exp\Big\{\dsl_{z \in S_1(0,m)\backslash \cA_m} (1 + c\,d^{-(1 - \ve)}) \,u_0 \dis\int_{W^*_z} (e^{\lambda f_k} - 1) d\nu\Big\}\,,
\end{equation}

\n
when $1 \le m \le k$, and a similar formula with $\{0\}$ in place of $S_1(0,m) \backslash \cA_m$, when $m=0$.

\medskip
We will now bound the integral in the right-hand side of (\ref{3.39}). In view of (\ref{1.31}) it equals:
\begin{equation}\label{3.40}
\begin{array}{l}
E^{Q_z} \big[  |X_n|_1 > |z|_1, \;\mbox{for} \; n < 0, |X_n|_1 \ge |z|_1,\; \mbox{for} \; n \ge 0, \;
e^{\lambda \sum\limits_{n \in \IZ} 1\{X_n \in \cA_k\}} - 1\big] =
 \\[1ex]
 a_- a_+ + a_- \,P_z[|X_n|_1 \ge |z|_1, \;\mbox{for} \;n \ge 0] + P_z[|X_n|_1 > |z|_1, \;\mbox{for} \;n > 0] \,a_+\,,
\end{array}
\end{equation}
where we have set
\begin{equation}\label{3.41}
\begin{split}
a_- & = E_z \big[|X_n|_1 > |z|_1, \;\mbox{for} \; n > 0, \;e^{\lambda \sum_{n > 0} 1\{X_n \in \cA_k\}} - 1\big],
\\[1ex]
a_+ & = E_z \big[|X_n|_1 \ge |z|_1, \;\mbox{for} \; n \ge 0, \;e^{\lambda \sum_{n \ge 0} 1\{X_n \in \cA_k\}} - 1\big],
\end{split}
\end{equation}

\n
and used the identity $u v-1 = (u-1) (v-1) + (u-1) + (v-1)$, together with (\ref{1.32}), (\ref{1.24}). Applying the strong Markov property at time $\wt{H}_{\cA_k}$, and denoting with $\psi$ the same function as below (\ref{1.3}), with the choice $K = \cA_k$,  we find that, when $d \ge c$, for $0 \le m \le k$, $1 \le k \le \ell$, $z \in S_1 (0,m) \backslash \cA_m$, when $m\not=0$ and $z=0$, when $m = 0$, one has:
\begin{equation}\label{3.42}
\begin{array}{l}
a_- \le P_z[|X_n|_1 > |z|_1, \;\mbox{for} \;0 < n \le \wt{H}_{\cA_k}, \wt{H}_{\cA_k} < \infty] (|\psi |_\infty - 1) \stackrel{(\ref{1.12})}{\le}
\\[1ex]
P_z\big[|X_n|_1 > |z|_1, \;\mbox{for} \;0 < n \le \wt{H}_{\cA_k}, \,\wt{H}_{\cA_k} < \infty \,,
\\[1ex]
P_{X_{\wt{H}_{A_k}}} [|X_n|_1 > |X_0|_1, \;\mbox{for} \;n > 0]\big] (|\psi|_\infty - 1) \Big(1 + \mbox{\f $\dis\frac{c}{d}$}\Big) \le
\\[1ex]
P_z\big[|X_n|_1 > |z|_1, \;\mbox{for all} \; n> 0, \;\wt{H}_{\cA_k} < \infty] (|\psi|_\infty - 1) \Big(1 + \mbox{\f $\dis\frac{c}{d}$}\Big)\,.
\end{array}
\end{equation}

\n
In the same fashion we also have
\begin{equation}\label{3.43}
a_+ \le P_z [|X_n|_1 \ge |z|_1, \;\mbox{for all} \; n \ge 0, \,H_{\cA_k} < \infty] (|\psi|_\infty - 1) \Big(1 + \mbox{\f $\dis\frac{c}{d}$}\Big)\,.
\end{equation}

\n
Inserting these bounds in the last line of (\ref{3.40}), we see that when $d \ge c$, $0 \le m \le k$, $1 \le k \le \ell$, and $z$ as above (\ref{3.42}), one has
\begin{equation}\label{3.44}
\dis\int_{W^*_z} (e^{\lambda f_k} - 1) d \nu \le c \,\nu(W^*_z \cap W^*_{\cA_k}) [(|\psi|_\infty - 1)^2 + |\psi|_\infty - 1]\,.
\end{equation}
With Lemma \ref{lem1.1} and (\ref{1.12}), we find that for $d \ge c$, when $\lambda \le 1$,
\begin{equation}\label{3.45}
|\psi|_\infty - 1 \le (e^\lambda - 1)\Big(1 + \mbox{\f $\dis\frac{c}{d}$}\Big) \le c \,\lambda\,,
\end{equation}

\n
so that for $1 \le m \le k$, $1 \le k \le \ell$, and $0 \le \lambda \le 1$, we have:
\begin{equation}\label{3.46}
\begin{array}{l}
\dsl_{z \in S_1(0,m) \backslash \cA_m} \;\dis\int_{W^*_z} (e^{\lambda f_k} - 1) d\nu \le c\,\lambda \,\nu\Big(\Big(\bigcup\limits_{z \in S_1(0,m) \backslash \cA_m} W^*_z\Big) \cap W^*_{\cA_k}\Big)
\\
\stackrel{(\ref{1.33})}{\le} c \,\lambda \,\cE(S_1 (0,m) \backslash \cA_m, \cA_k)\,,
\end{array}
\end{equation}

\n
and a similar inequality with $\{0\}$ in place of $S_1(0,m) \backslash \cA_m$, when $m = 0$. 

\medskip
When $m < k$, then we see with (\ref{1.33}) that
\begin{equation}\label{3.47}
\begin{array}{rrll}
{\rm i)} & \cE(S_1(0,m)  \backslash \cA_m, \cA_k)&\hspace{-1ex} \le&\hspace{-2ex} |\cA_k| \sup\limits_{x \in \cA_k} \,P_x[H_{S_1(0,m)  \backslash \cA_m} < \infty] 
\\
&&\hspace{-1ex}\!\!\! \stackrel{(\ref{1.12})}{\le}&\hspace{-2ex} \mbox{\f $\dis\frac{c}{d}$} \;| \cA_k|, \;\mbox{when} \; m \not= 0 \,.
\\[3ex]
{\rm ii)} &  \cE(\{0\}, \cA_k) &\hspace{-1ex}\le&\hspace{-2ex} \mbox{\f $\dis\frac{c}{d}$} \;| \cA_k|,\;\mbox{when $m = 0$, (in an analogous fashion)}.
\end{array}
\end{equation}

\n
On the other hand when $m = k$, we have
\begin{equation}\label{3.48}
\begin{array}{l}
\cE(S_1(0,k) \backslash \cA_k, \cA_k) \le |\cA_k| (\sup\limits_{\cA_k} \,P_x[H_{S_1(0,k-1)} < \infty] \; +
\\
 \sup\limits_{\cA_k} \,P_x[H_{S_1(0,k)} \circ \theta_{H_{S_1(0,k+1)}} < \infty])\stackrel{(\ref{1.12})}{\le} \mbox{\f $\dis\frac{c}{d}$} \;|\cA_k|\,.
\end{array}
\end{equation}

\n
Coming back to (\ref{3.27}), (\ref{3.39}), (\ref{3.46}), we see that for $d \ge c$, $0 \le \lambda \le 1$, $\cA \in \cS$, $0 \le m \le \ell$, $1 \le k \le \ell$, we have
\begin{equation}\label{3.49}
\begin{array}{l}
\IP[N_{m,k} \ge (\ell + 1)^{-1} \,d^{-\ve/2} |\cA_k| \,| \cT^0 = \cA] \le
\\
\exp\{ \lambda | \cA_k| ( - (\ell + 1)^{-1} \,d^{-\ve/2} + cd^{-1})\} \le \exp\Big\{ - \dis\frac{\lambda \,d^{-\ve/2}}{2(\ell + 1)} \;|\cA_k|\Big\}\,.
\end{array}
\end{equation}

\n
Choosing $\lambda = 1$, and taking into account that $|\cA_k| \ge c\,d^\ve$ due to (\ref{3.20}), (\ref{3.17}), we find (\ref{3.26}). The claim of Lemma \ref{lem3.3} then follows.
\end{proof}

We thus proceed with the proof of Theorem \ref{theo3.1} and specifically with the proof of (\ref{3.6}). We introduce the subtree $\cT^\prime$ of $\cT^0$, which survives the pruning due to the damage caused by $\cI^{u_0}$. More precisely, writing $\cT_k^\prime = \cT^\prime \cap S_1 (0,k)$, for $0 \le k \le m$, we define
\begin{equation}\label{3.50}
\cT^\prime_0 = \cT^0_0 = \{0\}, \;\mbox{and}\;  \cT^\prime_k = \{x \in \cT^0_k; \; x \in \partial \,\cT^\prime_{k-1} \backslash \cI^{u_0}\}, \;\mbox{for $1 \le k \le \ell$}\,.
\end{equation}

\n
As a result of Lemma \ref{lem3.3}, we see that on the event $\cD$, except maybe on a a set of probability at most $c \,e^{-cd^{\ve/2}}$ we have:
\begin{equation}\label{3.51}
|\cT^\prime_k| \ge  \; e^{-\frac{10^{-3}}{\ell}} \; \mbox{\f $\dis\frac{d^\ve}{\ell}$} | \cT^\prime_{k-1}| - d^{-\ve/2} |\cT^0_k|, \;\mbox{for $1 \le k \le \ell$} \,.
\end{equation}

\n
As we now explain, when $d \ge c$, these relations and the definition of $\cD$ in (\ref{3.17}) imply that
\begin{equation}\label{3.52}
|\cT^\prime_k| \ge e^{-\frac{k}{100 \ell}} |\cT^0_k |, \;\mbox{for} \;0 \le k \le \ell \,.
\end{equation}

\n
Indeed this is true for $k=0$. If it holds for $k-1 < \ell$, then with (\ref{3.51}) we find that,
\begin{equation*}
|\cT^\prime_k| \ge e^{- \frac{10^{-3}}{\ell}} \; 
\mbox{\f $\dis\frac{d^\ve}{\ell}$}  \;  e^{-\frac{k-1}{100 \ell}} |\cT^0_{k-1}| - d^{-\ve/2} |\cT^0_k| \stackrel{{\rm on} \;\cD}{\ge} \;\big(e^{-\frac{k-1}{100 \ell}} \,e^{-\frac{1}{500 \ell}} -d^{-\ve/2}\big) | \cT^0_k | \stackrel{d\ge c}{\ge} e^{-\frac{k}{100 \ell}} |\cT^0_k|\,,
\end{equation*}
and this proves (\ref{3.52}) by induction.

\medskip
As a result we have shown that on a set of probability at least $1 - c \,e^{-cd^{\ve/2}}$, the subtree $\cT^\prime$ of $\cT^0 \subseteq C$ satisfies:
\begin{equation}\label{3.53}
\cT^\prime \cap \cI^{u_0} \subseteq \{0\}, \;|\cT^\prime_\ell | \ge e^{-\frac{1}{100} - \frac{1}{1000}} \;\Big(\mbox{\f $\dis\frac{d^\ve}{\ell}$} \Big)^\ell , \;\mbox{and} \; |\cT^\prime_1| \le e^{\frac{10^{-3}}{\ell}} \;\mbox{\f $\dis\frac{d^\ve}{\ell}$}\,.
\end{equation}

\n
Thus when $d \ge c$, at least one element of $\cT^\prime_1$ belongs to a component of $\cV^{u_0} \cap C$ with at least $L = [\frac{1}{2} (\frac{d^\ve}{\ell})^{\ell -1}]$ elements, and (\ref{3.7}) follows. Theorem \ref{theo3.1} is proved.
\end{proof}

\section{Sprinkling}
\setcounter{equation}{0}

In this section we will complete the proof of Theorem \ref{theo0.1}. The strategy is to check that assumption (\ref{2.4}) of Theorem \ref{theo2.2} holds when we choose $u(d) = u_2 (< u_0)$, in the notation of (\ref{4.1}) below. We rely on Theorem \ref{theo3.1} of the previous section and use two successive sprinkling operations in order to check (\ref{2.4}) relative to $u_2$. The first sprinkling operation amounts to working with $\cV^{u_1}$ in place of $\cV^{u_0}$, where $u_2 < u_1 < u_0$, see (\ref{4.1}). We show in Theorem \ref{theo4.2} that with high probability $\cV^{u_1} \cap C$ contains a ubiquitous component in the sense of Lemma \ref{lem2.1}. Replacing the level $u_0$ with the level $u_1$ amounts to ``sprinkling more vacant sites''. This is achieved by thinning the trajectories of the random interlacement at level $u_0$ with the help of i.i.d. Bernoulli variables, cf.~(\ref{4.14}), (\ref{4.15}). As in \cite{AlonBenjStac04}, \cite{Gord91}, we use isoperimetry considerations to ensure that this first sprinkling typically merges together any sizeable union of substantial components of $\cV^{u_0} \cap C$, with volume at least $d^{-4} |C|$, cf.~(\ref{4.27}), and thus creates the desired ubiquitous component of $\cV^{u_1} \cap C$, cf.~Theorem \ref{theo4.2}. The second sprinkling amounts to working with $\cV^{u_2}$ in place of $\cV^{u_1}$ and ensures that with high probability the ubiquitous components of $\cV^{u_1}$ in $C_y$, for $|y|_1 \le 1$, $y$ in $\IZ^2$, are mutually connected at level $u_2$, and (\ref{2.4}) holds for $u_2$, cf.~Theorem \ref{theo4.4}. Although the general line of attack is similar to \cite{AlonBenjStac04}, \cite{Gord91}, the long range dependence of random interlacements makes the implementation of this general strategy different from what is done in the case of Bernoulli percolation. As in the previous section all constants depend on the parameter $\ve$, see (\ref{3.1}). We first introduce some notation. We set, see (\ref{3.2}), (\ref{2.1}), 
\begin{align}
u_2 & = (1- 3 \ve) \,g(0) \,\log d < u_1 = (1-2 \ve) \,g(0) \,\log d < u_0, \; \mbox{and}\label{4.1}
\\[1ex]
\wh{C} & = \bigcup\limits_{|y|_1 \le 1, y \in \IZ^2} C_y\,. \label{4.2}
\end{align}

\n
As a preparation for the sprinkling operations we intend to perform, which involve thinning out trajectories of the random interlacement, we first bound the number of trajectories going through a point of $\wh{C}$ and the time spent in $\wh{C}$ by trajectories of the interlacement meeting $\wh{C}$. We recall (\ref{1.34}) and (\ref{3.3}) for the notation.

\begin{lemma}\label{lem4.1}
\begin{align}
&\IP[\mbox{for some} \;|y|_1 \le 1 \;{\rm in} \; \IZ^2, \;|\{x \in C_y; \,\mu_{\{x\},u_0} (W_+) \ge \alpha \,\log d\} | \ge d^{-7} |C|] \le d^{-7} / 100 \label{4.3}
\\[2ex]
&\IP\Big[\mu_{\wh{C},u_0} \big(\big\{ w \in W_+; \;\dsl_{n \ge 0} 1\{X_n(w) \in \wh{C}\} \ge c_3 \,d\big\}\big) > 0\Big] \le c\,e^{-c\,d}\,. \label{4.4}
\end{align}
\end{lemma}

\begin{proof}
The probability in (\ref{4.3}) is smaller than
\begin{align*}
&5 \IP[|\{x \in C; \mu_{\{x\},u_0}(W_+) \ge \alpha \log d\}| \ge d^{-7}\,|C|] \stackrel{Chebyshev}{\le}
\\[1ex]
&5d^7 \,\IP[\mu_{\{0\},u_0} (W_+) \ge \alpha \, \log d]  \stackrel{(\ref{1.35}), (\ref{1.26}), (\ref{3.4})}{\le} d^{-7}/100 \,.
\end{align*}

\medskip\n
We now turn to the proof of (\ref{4.4}). Since we assume that $d \ge 3$, it follows that
\begin{equation}\label{4.5}
\inf\limits_{z \in C} \;P_z[\wt{H}_C = \infty] \ge c > 0\,.
\end{equation}

\n
As a result of Lemma \ref{lem1.1}, we see that for a small enough constant $c$,
\begin{equation*}
\sup\limits_{z \in \IZ^d} \;E_x \Big[\exp\Big\{c\, \dsl_{n \ge 0} \,1\{X_n \in C\}\Big\}\Big] \le 2\,.
\end{equation*}

\n
Setting $c^\prime = c/5$, an application of H\"older's inequality thus yields 
\begin{equation}\label{4.6}
\sup\limits_{x \in \IZ^d} \;E_x\Big[\exp\Big\{c^\prime\mbox{$ \dsl_{n \ge 0}$} 1\{X_n \in \wh{C}\}\Big\}\Big] \le 2\,.
\end{equation}

\n
With (\ref{1.35}), we find that for $\lambda, \rho > 0$, and $c^\prime$ as in (\ref{4.6}) we have
\begin{equation}\label{4.7}
\begin{array}{l}
\IP \Big[\mu_{\wh{C},u_0} \big(\big\{w \in W_+; \dsl_{n \ge 0} 1\{X_n(w) \in \wh{C}\} \ge \rho\big\}\big) > 0\Big] \le \IP[\mu_{\wh{C},u_0}(W_+) \ge 2u_0 | \wh{C}|] \; +
\\[1ex]
2 u_0 |\wh{C} | \;\sup\limits_{x \in \IZ^d} \,P_x\Big[\dsl_{n \ge 0} 1\{X_n(w) \in \wh{C} \} \ge \rho\Big]  \stackrel{(\ref{1.35}),(\ref{4.6})}{\le}
\\[3ex]
\exp\{u_0 [{\rm cap}(\wh{C}) (e^\lambda -1) - 2\lambda \,|\wh{C}|]\} + 4 u_0 |\wh{C} | \,\exp\{ -c^\prime \rho\}\,.
\end{array}
\end{equation}

\n
Since ${\rm cap}(\wh{C}) \le |\wh{C}| = 5 \cdot 2^d$, choosing $\lambda$ a small enough constant and $\rho = c_3\,d$, with $c_3$ a large enough constant, we obtain (\ref{4.4}). 
\end{proof}

We will now conduct the first sprinkling. For this purpose we introduce a coupling of $\cV^{u_0} \cap C$ and $\cV^{u_1} \cap C$ as follows. On an auxiliary probability space $(\wt{\Omega}, \wt{\cA}, \wt{\IP})$, we consider independent variables
\begin{align}
&\mbox{$\wt{N}$, with Poisson distribution of intensity $u_0\, {\rm cap}(C)$}, \label{4.8}
\\[1ex]
&\mbox{$\wt{X}_\point^i, i \ge 1$, i.i.d. $W_+$-valued variables with distribution $P_{\ov{e}_C}$}, \label{4.9}
\\[-1ex]
&\mbox{where $\ov{e}_C = e_C / {\rm cap}(C)$}, \nonumber
\\[1ex]
& \mbox{$\wt{Z}^i, i \ge 1$, i.i.d. Bernoulli variables with success probability $\mbox{\f $\dis\frac{u_1}{u_0} \;\Big(= \dis\frac{1-2\ve}{1-\ve}\Big)$}$}\,. \label{4.10}
\end{align}
We then define the sub $\sigma$-algebra of $\wt{\cA}$:
\begin{equation}\label{4.11}
\wt{\cA}_0 = \sigma(\wt{N}, \wt{X}_\point^i, i \ge 1)\,,
\end{equation}
and the point processes on $W_+$:
\begin{equation}\label{4.12}
\wt{\mu}_0(dw) = \dsl_{1 \le i \le \wt{N}} \delta_{\wt{X}_\point^i} (dw), \quad \wt{\mu}_1(dw) = \dsl_{1 \le i \le \wt{N}} \wt{Z}^i \,\delta_{\wt{X}_\point^i} (dw)\,.
\end{equation}

\n
With (\ref{1.35}) and (\ref{4.12}) we thus see that
\begin{equation}\label{4.13}
\mbox{$\wt{\mu}_0$ is $\wt{\cA}_0$-measurable with the same distribution as $\mu_{C,u_0}$ under $\IP$}\,.
\end{equation}
Moreover by construction,
\begin{equation}\label{4.14}
\mbox{$\wt{Z}^i, i \ge 1$, are independent of $\wt{\cA}_0$},
\end{equation}

\n
and $\wt{\mu}_1$ is obtained by thinning $\wt{\mu}_0$ with the variables $\wt{Z}_i, i \ge 1$, so that:
\begin{equation}\label{4.15}
\mbox{$\wt{\mu}_1$ has the same distribution as $\mu_{C,u_1}$ under $\IP$}\,.
\end{equation}

\medskip\n
Out next main objective is the following

\begin{theorem}\label{theo4.2} $(d \ge c)$
\begin{equation}\label{4.16}
\begin{array}{l}
\mbox{$\IP\big[\cV^{u_1} \cap C$ has a connected component $\cC$ with} 
\\[0.5ex]
|\ov{\cC} \cap C| \ge (1-d^{-2}) | C |\big] \ge 1-d^{-7}/10,
\end{array}
\end{equation}

\medskip\n
(we recall that when $d \ge c^\prime$, $\cC$ is necessarily unique, cf.~Lemma {\rm \ref{lem2.1}}, and we refer to this component as the ubiquitous component of $\cV^{u_1} \cap C$).
\end{theorem}

\begin{proof}
We use the auxiliary probability space $(\wt{\Omega}, \wt{\cA}, \wt{\IP})$ and introduce
\begin{equation}\label{4.17}
\wt{\cV}_i = C\; \backslash \bigcup\limits_{w \in {\rm Supp}\,\wt{\mu}_i} w(\IN), \; \mbox{for $i = 0,1$}\,.
\end{equation}

\n
In view of (\ref{1.38}), (\ref{4.13}), (\ref{4.15}), we find that for $i = 0,1$,
\begin{equation}\label{4.18}
\mbox{$\wt{\cV}_i$ under $\wt{\IP}$ has the same distribution as $\cV^{u_i} \cap C$ under $\IP$}\,.
\end{equation}

\n
We thus only need to consider $\wt{\cV}_1$ when proving (\ref{4.16}). We introduce the good event, see (\ref{3.6}), (\ref{4.3}), (\ref{4.4}),
\begin{equation}\label{4.19}
\begin{split}
\wt{\cG} =  \;&\{ \wt{\o} \in \wt{\Omega}; \;\mbox{for all but at most $|C|\,e^{-c_2 d^{\ve/2}}$ points, the vertices in $C$ have}
\\
&\;\mbox{a neighbor in a substantial component of $\wt{\cV}_0\} \cap \{\wt{\o} \in \wt{\Omega}$; for all but}
\\
&\;\mbox{at most $d^{-7} |C|$ points, the vertices in $C$ belong to the range of at}
\\
& \; \mbox{most $\alpha \log d$ trajectories in ${\rm Supp} \,\wt{\mu}_0\} \cap \{\wt{\omega} \in \wt{\Omega}$; all trajectories in} 
\\
&\;\mbox{${\rm Supp}\,\wt{\mu}_0$ spend a time at most $c_3 \,d$ in $C\}$}.
\end{split}
\end{equation}
It is plain that
\begin{equation}\label{4.20}
\mbox{$\wt{\cG}$ is $\wt{\cA}_0$-measurable, (and hence independent of $\wt{Z}_i$, $i \ge 1$)}\,.
\end{equation}

\n
Thanks to Theorem \ref{theo3.1} and Lemma \ref{lem4.1}, $\wt{\cG}$ has overwhelming probability:
\begin{equation}\label{4.21}
\wt{\IP} [ \wt{\cG}^c] \le c\,e^{-c\,d^{\ve/2}} + d^{-7}/100 + c\,e^{-c\,d}\,.
\end{equation}
We then introduce the random subset of bad points of $C$:
\begin{align}
\cM(\wt{\o})  = &\; \{x \in C; \; \mbox{no neighbor of $x$ lies in a substantial component of $\wt{\cV}_0(\wt{\o})\}$}\; \cup\label{4.22}
\\
& \;\{ x \in C; \;\mbox{more than $\alpha \log d$ trajectories in ${\rm Supp} \,\wt{\mu}_0(\wt{\o})$ enter $x\}$}\,. \nonumber
\end{align}

\n
From this definition and with (\ref{4.19}), we find that
\begin{equation}\label{4.23}
\mbox{$\cM$ is $\wt{\cA}_0$-measurable, and when $d \ge c$, $|\cM| \le 2d^{-7} |C|$, on $\wt{\cG}$}\,.
\end{equation}

\n
For the sprinkling argument, it will be convenient to specify $\wt{\cV}_0$ and $\cM$. To this effect we consider the $\wt{\cA}_0$-measurable events
\begin{equation}\label{4.24}
\wt{\cG}_{V,M} = \{ \wt{\o} \in \wt{\cG}; \,\wt{\cV}_0 = V, \;\cM = M\} \;\mbox{for $V, M \subseteq C$, such that $\wt{\cG}_{V,M} \not= \emptyset$},
\end{equation}

\n
(note that necessarily $|M| \le 2d^{-7} |C|$).

\medskip
Given $V,M$ as above, a connected component of $V$ containing at least $L$ points will be called {\it substantial}, (this follows the terminology introduced above (\ref{3.6})). We denote with $\cC_1^V, \dots, \cC^V_{\ell_V}$ the substantial components of $V$. Due to the fact that on $\wt{\cG}_{V,M}$ any vertex in $C \backslash M$ is a neighbor of a substantial component, we see that for $d \ge c$, with $c$ as in (\ref{4.23}),
\begin{equation}\label{4.25}
|\cC_1^V \cup \dots \cup \cC^V_{\ell_V} | \ge d^{-1} |C \backslash M | \stackrel{(\ref{4.23})}{\ge} d^{-1} ( 1-2d^{-7}) \,|C|\,,
\end{equation}
and moreover one has:
\begin{equation}\label{4.26}
\ell_V \,L \le |C| ( = 2^d)\,.
\end{equation}

\n
Given $V,M$ as in (\ref{4.24}), we consider a partition of $\{1,\dots, \ell_V\}$ into two subsets $I_A$ and $I_B$, where with a slight abuse of notation, $A$ and $B$ respectively stand for the union of the $\cC^V_i$, where $i$ respectively runs over $I_A$ and $I_B$. We assume $A$ and $B$ sizeable in the sense that
\begin{equation}\label{4.27}
|A| \wedge |B| \ge d^{-4} |C|\,.
\end{equation}

\n
In the notation of (\ref{1.2}), it follows from the definition of $A,B$ that
\begin{equation}\label{4.28}
A \cap \ov{B}\,^C = \emptyset = \ov{A}\,^C \cap B\,.
\end{equation}

\n
Moreover with the observation leading to (\ref{4.25}) and with (\ref{4.23}), we have
\begin{equation}\label{4.29}
\ov{A}\,^C \cup \ov{B}\,^C \cup M = C, \;\mbox{and} \; |M| \le 2d^{-7} |C|\,.
\end{equation}

\n
As we will now see, isoperimetry considerations for subsets of $C$ yield the existence of a wealth of disjoint paths of length at most $3$ between $A$ and $B$ above.

\begin{lemma}\label{lem4.3} $(d \ge c)$

\medskip
Consider $A,B,M$ subsets of $C$ satisfying {\rm (\ref{4.27}) - (\ref{4.29})}. Then there are $K$ pairwise disjoint sets $\cP_1,\dots , \cP_K$ in $C \backslash M$, with
\begin{equation} \label{4.30}
K \ge c\,d^{-6} |C|\,,
\end{equation}
and
\begin{align}
&\mbox{for $1 \le k \le K$, $\cP_k$ is either of the form $\cP_k = \{x_k\}$ with $x_k \in \partial_C A \cap \partial_C B$,}\label{4.31}
\\
&\mbox{or of the form $\cP_k = \{x_k,y_k\}$, with $x_k \in \partial_C A, y_k \in \partial_C \,B$ and $|x_k - y_k|_1 = 1$}\,. \nonumber
\end{align}
\end{lemma}

\begin{proof}
We write for simplicity $\ov{F}$ and $\partial F$ in place of $\ov{F}\,^C$ and $\partial_C\,F$, for $F \subseteq C$, throughout the proof of Lemma \ref{lem4.3}.

\medskip
Assume first that $|\ov{A} \cap \ov{B}| \ge d^{-6} |C|$. In this case we find that\begin{equation}\label{4.32}
|(\partial A \cap \partial B) \backslash M| \stackrel{(\ref{4.28})}{=} |\ov{A} \cap \ov{B} \backslash M | \ge (d^{-6} - 2d^{-7}) \ge \fr  \;d^{-6} |C|, \;\mbox{for} \; d \ge c\,.
\end{equation}

\n
This proves (\ref{4.30}), (\ref{4.31}), with $\cP_i$, $1 \le i \le K$, a collection of singletons.

\medskip
Assume now that $|\ov{A} \cap \ov{B}| < d^{-6} |C|$. We can assume without loss of generality that $|\ov{A}| \le |\ov{B}|$. As a result $2 |\ov{A}| \le |\ov{A}| + |\ov{B}| \le |\ov{A} \cup \ov{B}| + |\ov{A} \cap \ov{B}| \le ( 1 + d^{-6})|C|$, so that:
\begin{equation}\label{4.33}
|\ov{A}| \le \mbox{\f $\dis\frac{3}{4}$} \;|C|\,.
\end{equation}

\n
It then follows from the isoperimetric controls in Corollary 4, p.~305 of Bollobas-Leader \cite{BollLead91}, that there are at least $c\,|\ov{A}|$ distinct edges between $\ov{A}$ and $C \backslash \ov{A}$. As a result we see that
\begin{equation}\label{4.34}
|\partial \ov{A} | \ge c\, d^{-1} |\ov{A}| \stackrel{(\ref{4.27})}{\ge} c\,d^{-5} |C| \,.
\end{equation}

\medskip\n
From the bound on $|M|$ in (\ref{4.29}), we infer that for $d \ge c$,
\begin{equation}\label{4.35}
|\partial \ov{A} \backslash M| \ge (c\,d^{-5} - 2d^{-7}) |C| \ge c^\prime d^{-5} |C|\,.
\end{equation}

\n
In view of (\ref{4.29}), we find that $\partial \ov{A} \backslash M \subseteq \ov{B}$, and we can therefore select at least $c^\prime \,d^{-6} |C|$ pairwise disjoint sets $\{x,y\}$ with $x \in \ov{A}$, $y \in \ov{B} \backslash (\ov{A} \cup M)$, and $|x-y|_1 = 1$. Eliminating the pairs for which $x \in M$, leaves a collection of at least $(c^\prime\,d^{-6} - 2d^{-7}) |C| \ge c\,d^{-6} |C|$ pairwise disjoint sets $\{x,y\}$, and $x \in \ov{A} \backslash M$, $y \in \ov{B} \backslash (\ov{A} \cup M)$, and $|x-y|_1 = 1$. Note that necessarily $x \in \partial A \backslash M$, (otherwise $y$ would belong to $\ov{A}$). Dropping $y$ from the set when $y \in B$, so that $x \in (\partial A \cap \partial B) \backslash M$, we have the desired collection $\cP_k$, $1 \le k \le K$. This concludes the proof of Lemma \ref{lem4.3}.
\end{proof}

We resume the proof of Theorem \ref{theo4.2}. We assume $d \ge c$, so that (\ref{4.23}) and the above lemma hold. Given $V,M$ as in (\ref{4.24}), for any fixed $A,B$, as described below (\ref{4.26}), such that (\ref{4.27}) holds, we select a collection $\cP_1,\dots, \cP_K$ of pairwise disjoint sets of $C \backslash M$ satisfying (\ref{4.30}), (\ref{4.31}).

\medskip
On the event $\wt{\cG}_{V,M}$ of (\ref{4.24}), each $\cP_k$ meets at most $2 \alpha \log d$ trajectories in ${\rm Supp} \, \wt{\mu}_0$, and each trajectory in ${\rm Supp} \,\wt{\mu}_0$ meets at most $c_3\,d$ sets $\cP_{k^\prime}$, $1 \le k^\prime \le K$. We say that $k,k^\prime$ in $\{1,\dots ,K\}$ are related if some trajectory in ${\rm Supp} \, \wt{\mu}_0$ meets $\cP_k$ and $\cP_{k^\prime}$. It is plain that for $1 \le k, k^\prime \le K$,
\begin{align}
&\{\wt{\o} \in \wt{\cG}_{V,M}; \; k \;\mbox{and} \; k^\prime \;\mbox{are related}\} \in \wt{\cA}_0, \;\mbox{and} \label{4.36}
\\[2ex]
&\mbox{on $\wt{\cG}_{V,M}$ each $k$ is related to at most $R = 2 \alpha \,c_3\,d\,\log d$} \label{4.37}
\\
&\mbox{integers $k^\prime$ in $\{1,\dots,K\}$} \,.\nonumber
\end{align}

\n
We then construct for $\wt{\o} \in \wt{\cG}_{V,M}$ in a straightforward fashion an $\wt{\cA}_0$-measurable random subset $\cK (\wt{\o})$ of $\{1,\dots,K\}$, consisting of unrelated integers that is maximal. We pick $k_1 = 1$. If all integers in $\{1,\dots,K\} \backslash \{k_1\}$ are related to $k_1$, we stop. Otherwise, we pick $k_2$ the smallest $k$ in $\{1,\dots, K\} \backslash \{k_1\}$ unrelated to $k_1$. If all integers are related to $k_1$ or $k_2$, we stop. Otherwise we pick $k_3$, the smallest integer unrelated to $k_1$ or $k_2$, and so on, until we stop at some point. In this way for $\wt{\o} \in \wt{\cG}_{V,M}$ we construct $\cK(\wt{\o}) \subseteq \{1, \dots, K\}$, such that 
\begin{equation}\label{4.38}
\begin{array}{rl}
{\rm i)} & \mbox{$\cK(\wt{\o})$ is $\wt{\cA}_0$-measurable}\,,
\\[1ex]
{\rm ii)} & |\cK(\wt{\o})| \ge [K / R] \ge c(d^7 \,\log d)^{-1} |C|\,,
\\[1ex]
{\rm iii)} & \mbox{for $k \not= k^\prime$ in $\cK(\o)$, $k$ and $k^\prime$ are unrelated}\,.
\end{array}
\end{equation}

\n
We will now see that on $\wt{\cG}_{V,M}$, the conditional probability given $\wt{\cA}_0$, that $A$ and $B$ are in the same connected component of $\wt{\cV}_1$, is overwhelming. Indeed on $\wt{\cG}_{V,M}$ the sets of trajectories in ${\rm Supp} \, \wt{\mu}_0$, which enter the various $\cP_k$, $k \in \cK$, are pairwise disjoint. If some $k \in \cK$ is such that all Bernoulli variables $\wt{Z}^i$, with $i \ge 1$, varying over the set of indexes of the trajectories $\wt{X}_\point^i$ meeting $\cP_k$, equal zero, then $\cP_k \subseteq \wt{\cV}_1$, and in view of (\ref{4.31}), $A$ and $B$ belong to the same connected component of $\wt{\cV}_1$. With (\ref{4.10}), (\ref{4.14}), we see that on $\wt{\cG}_{V,M}$, for $A,B$ as described below (\ref{4.26}), $\wt{\IP}$-a.s.,
\begin{equation}\label{4.39}
\begin{array}{l}
\mbox{$\wt{\IP}[A,B$ are not connected in $\wt{\cV}_1 | \wt{\cA}_0] \le \Big(1 - \Big(\mbox{\f $\dis\frac{u_0 - u_1}{u_0}$}\Big)^{2 \alpha \log d}\Big)^{[K/R]} \stackrel{(\ref{4.1}),(\ref{4.38})\,{\rm ii)}}{\le}$}
\\
\exp\{-c\,\ve^{2 \alpha \log d}(d^7 \log d)^{-1} |C|\}\,.
\end{array}
\end{equation}

\n
The number of possible choices for $A,B$ is at most $2^{\ell_V}$. Therefore when $d \ge c$, for $V,M$ satisfying (\ref{4.24}), we see that $\wt{\IP}$-a.s. on $\wt{\cG}_{V,M}$,
\begin{equation}\label{4.40}
\begin{array}{l}
\wt{\IP}\Big[\dis\bigcup\limits_{A,B} \{A,B \;\mbox{are not connected in} \;\wt{\cV}_1\}|  \wt{\cA}_0\Big]  \stackrel{(\ref{3.5}),(\ref{4.26}),(\ref{4.39})}{\le}
\\
\exp\{c\,d^{-\ve(\ell -1)} |C| - c\,d^{-(7 + 2 \alpha \log \frac{1}{\ve})} (\log d)^{-1} |C|\}
\stackrel{(\ref{3.5})}{\le} d^{-7}/100\,,
\end{array}
\end{equation}

\n
where $A,B$ in (\ref{4.40}) run over aggregate unions of substantial components of $V$ corresponding to a partition of $\{1, \dots , \ell_V\}$ into two classes, such that $|A|$ and $|B|$ are at least $d^{-4} |C|$, cf.~below (\ref{4.26}).

\medskip
On the event $\wt{\cH}_{V,M}$, complement in $\wt{\cG}_{V,M}$ of the event that appears in (\ref{4.40}), we can consider the random partition of $\{1,\dots ,\ell_V\}$ induced by the various $\cC^V_i$, $1 \le i \le \ell_V$, which belong to the same connected component of $\wt{\cV}_1$. Denote with $I$ an element of this partition such that $|\bigcup_{i \in I} \cC^V_i |$ is maximal. When $d \ge c$, with (\ref{4.25}), and the definition of $\wt{\cH}_{V,M}$, we see that necessarily $|\bigcup_{i \in I} \,\cC^V_i | \ge d^{-4} |C|$, and therefore $|\bigcup_{i \in I^c} \cC_i^V | < d^{-4} |C|$. As a result we find that $|\overline{\bigcup_{i \in I^c} \cC_i^V} \cap C| \le 2d^{-3} |C|$, and hence with the argument above (\ref{4.25}) we have
\begin{equation}\label{4.41}
\begin{split}
\Big|\overline{\bigcup\limits_{i \in I} \,\cC_i^V} \cap C\Big| \ge |C| - \Big| \overline{\bigcup\limits_{i \in I^c} \cC^V_i} \cap C | - |M| & \ge (1-2d^{-3} - 2d^{-7}) |C|
\\[-1ex]
& \ge (1-d^{-2}) |C| \,.
\end{split}
\end{equation}

\n
Since the $\wt{\cA}_0$-measurable events $\wt{\cG}_{V,M}$ partition $\wt{\cG}$, cf.~(\ref{4.24}), we thus find that when $d \ge c$, the conditional probability given $\wt{\cG}$, that $\wt{\cV}_1$ has an ubiquitous component, is at least $1 - d^{-7}/100$. Together with (\ref{4.21}), we readily find (\ref{4.16}). This concludes the proof of Theorem \ref{theo4.2}.
\end{proof}

We now turn to the second sprinkling operation, which is simpler. Our intent is to link together the various ubiquitous components of $\cV^{u_1} \cap C_y$, for $|y|_1 \le 1$ in $\IZ^2$. This will enable us to verify condition (\ref{2.4}) of Theorem \ref{theo2.2} with the choice $u(d) = u_2$, see (\ref{4.1}), and Theorem \ref{theo0.1} will follow. With this in mind, we consider an auxiliary probability space $(\wh{\Omega}, \wh{\cA}, \wh{\IP})$ endowed with a collection of independent variables $\wh{N}, \wh{X}^i_\point, i \ge 1, \wh{Z}^i, i \ge 1$, as in (\ref{4.8}) - (\ref{4.10}), simply replacing $C$ with $\wh{C}$, see (\ref{4.2}), and $\frac{u_1}{u_0}$ with $\frac{u_2}{u_1}$, see (\ref{4.1}). We then introduce the sub $\sigma$-algebra of $\wh{\cA}$:
\begin{equation}\label{4.42}
\wh{\cA}_1 = \sigma (\wh{N}, \wh{X}_\point^i, i \ge 1)\,,
\end{equation}
as well as the point processes on $W_+$:
\begin{equation}\label{4.43}
\wh{\mu}_1(dw) = \dsl_{1 \le i \le \wh{N}} \delta_{\wh{X}^i}(dw), \;\;\wh{\mu}_2(dw) = \dsl_{1 \le i \le \wh{N}} \wh{Z}^i \,\delta_{\wh{X}^i_\point} (dw)\,.
\end{equation}

\medskip\n
Just as in (\ref{4.13}) - (\ref{4.15}), we find that:
\begin{align}
&\mbox{$\wh{\mu}_1$ is $\wh{\cA}_1$-measurable, with the same distribution as $\mu_{\wh{C}, u_1}$ under $\IP$}, \label{4.44}
\\ 
&\mbox{$\wh{Z}^i$, $i \ge 1$, are independent of $\wh{\cA}_1$},\label{4.45}
\\ 
&\mbox{$\wh{\mu}_2$, has the same distribution as $\mu_{\wh{C},u_2}$ under $\IP$}.\label{4.46}
\end{align}

\medskip\n
We recall the notation from (\ref{2.3}) and (\ref{4.1}). The next objective is the following 

\begin{theorem}\label{theo4.4} $(d \ge c)$
\begin{equation}\label{4.47}
\IP[\cG_{u_2}] \ge 1-d^{-7}\,.
\end{equation}
\end{theorem}

\begin{proof}
On the above auxiliary probability space, we define random subsets of $\wh{C}$ via:
\begin{equation}\label{4.48}
\wh{\cV}_i = \wh{C}\; \backslash \bigcup\limits_{w \in {\rm Supp} \,\wh{\mu}_i} w(\IN), \;\mbox{for $i = 1,2$}\,.
\end{equation}

\n
In view of (\ref{4.44}), (\ref{4.46}), we find that for $i = 1,2$,
\begin{equation}\label{4.49}
\mbox{$\wh{\cV}_i$ under $\wh{\IP}$ has the same distribution as $\cV^{u_i} \cap \wh{C}$ under $\IP$}\,.
\end{equation}

\n
For $y$ in $\IZ^2$ with $|\cdot |_1$-norm at most $1$, we denote with $\wh{\cG}_y$ the event
\begin{equation}\label{4.50}
\mbox{$\wh{\cG}_y = \{\wh{\o} \in \wh{\Omega}$; there is a component $\cC_y$ of $\wh{\cV}_1 \cap C_y$ ubiquitous in $C_y\}$},
\end{equation}

\medskip\n
(from Lemma \ref{lem2.1}, we know that when $d \ge c$, such a component is necessarily unique). With Theorem \ref{theo4.2} we find that for $d \ge c$,
\begin{equation}\label{4.51}
\wh{\IP}[\wh{\cG}] \ge 1 - d^{-7}/2, \;\mbox{where} \;\wh{\cG} \, \stackrel{\rm def}{=} \bigcap\limits_{y \in \IZ^2, |y|_1 \le 1} \wh{\cG}_y\,.
\end{equation}

\n
For each $|y|_1 = 1$, $y$ in $\IZ^2$, we denote with $F_y = C_y \cap \partial C$, the ``face'' of $C_y$ abutting $C$, and with $F_{y,0} = C \cap \partial C_y ( = F_y - y)$, the ``face'' of $C$ abutting $C_y$. On the event $\wh{\cG}$, for each such $y$, $\ov{\cC}_y$ covers at least $|F_y| - d^{-2} |C| = (1-2d^{-2}) |F_y| \ge \frac{3}{4} \,|F_y|$ sites of $F_y$ and $\ov{\cC}$ covers at least $\frac{3}{4} \,|F_{y,0}|$ sites of $|F_{y,0}|$, (with the notation $\cC = \cC_{y=0}$). We thus find that on $\wh{\cG}$, for each $|y|_1 = 1$, in $\IZ^d$,
\begin{equation}\label{4.52}
\begin{array}{l}
\mbox{there are $\frac{1}{2} |F_y| = \frac{1}{4} \,|C|$ disjoint sets $\{z_k,z_{k,0}\}$, with}
\\[1ex]
\mbox{$z_k \in \ov{\cC}_y \cap F_y$, $z_{k,0} \in \ov{\cC} \cap F_{y,0}$, and $|z_k - z_{k,0}|_1 = 1$}\,.
\end{array}
\end{equation}

\n
A similar sprinkling argument as for the proof of Theorem \ref{theo4.2} now yields, (compare with (\ref{4.39})), that when $d \ge c$, $\wh{\IP}$-a.s. on $\wh{\cG}$,
\begin{equation}\label{4.53}
\begin{array}{l}
\mbox{$\wh{\IP}[\cC_y$ and $\cC$ are in distinct components of $\wh{\cV}_2 | \wh{\cA}_1] \le$}
\\
\Big(1 - \Big(\dis\frac{u_1 - u_2}{u_1}\Big)^{2 \alpha \log d}\Big)^{[|C|/4R]} \le \exp\{- c\, \ve^{2 \alpha \log d} (d \log d)^{-1} |C|\} \le d^{-7}/8\,.
\end{array}
\end{equation}

\n
Combining this estimate with (\ref{4.51}), we find that for $d \ge c$,
\begin{equation}\label{4.54}
\IP [\cG_{u_2}] \ge 1 - 4\,d^{-7}/8 - d^{-7}/2 = 1-d^7\,,
\end{equation}
and Theorem \ref{theo4.4} is proved.
\end{proof}

We can now complete the proof of Theorem \ref{theo0.1}.

\bigskip\n
{\it Proof of Theorem \ref{theo0.1}}: It follows from Theorem \ref{theo4.4} and Theorem \ref{theo2.2}, that for any $0 < \ve < \frac{1}{10}$, $u_* \ge (1-3 \ve) g(0) \log d$, except maybe for finitely many values of $d$. This readily implies that $\liminf_d u_*/\log d \ge 1$, and thus completes the proof of Theorem \ref{theo0.1}. \hfill $\square$

\begin{remark}\label{rem4.5} \rm
It is natural to wonder whether Theorem \ref{theo0.1} captures the correct growth of $u_*$ for large $d$, in the sense that
\begin{equation}\label{4.55}
\mbox{$u_* \sim \log d$, when $d \rightarrow \infty$}\,.
\end{equation}

\n
This is indeed the case and we refer to \cite{Szni10a}, where the asymptotic upper bound on $u_*$ complementing (\ref{0.3}) can be found. In fact one may wonder whether a finer result actually holds, namely:
\begin{equation}\label{4.56}
\IP [0 \in \cV^{u_*}] = e^{-u_*/g(0)} \sim \mbox{\f $\dis\frac{1}{2d}$, as $d \r \infty$}\,.
\end{equation}

\n
This would signal a similar behavior as for Bernoulli percolation, cf.~\cite{AlonBenjStac04}, \cite{BollKoha94}, \cite{Gord91}, \cite{HaraSlad90}, \cite{Kest90}. In the case of the percolation of the vacant set left by random interlacements on $2d$-regular trees, the probability corresponding to the left-hand side of (\ref{4.56}) can be explicitly computed, and a relation such as (\ref{4.56}) is known to hold, cf.~Remark 5.3 of Teixeira \cite{Teix08b}. It is of course an interesting question to understand, whether for large $d$ random interlacements on $\IZ^d$ exhibit similarities to random interlacements on $2d$-regular trees, and if this is the case, explore the nature of these similarities. \hfill $\square$
\end{remark}


\begin{thebibliography}{10}

\bibitem{AlonBenjStac04}
N.~Alon, I.~Benjamini, and A.~Stacey.
\newblock Percolation on finite graphs and isoperimetric inequalities.
\newblock {\em Ann. Probab.}, 32(32A):1727--1745, 2004.

\bibitem{BollKoha94}
B.~Bollob\'as and Y.~Kohayakawa.
\newblock Percolation in high dimension.
\newblock {\em Europ. J. Combinatorics}, 15:113--125, 1994.

\bibitem{BollLead91}
B.~Bollob\'as and I.~Leader.
\newblock Edge-isoperimetric inequalities in the grid.
\newblock {\em Combinatorica}, 11(4):299--314, 1991.

\bibitem{Carn85}
T.K. Carne.
\newblock A transmutation formula for {M}arkov chains.
\newblock {\em Bull. Sc. Math.}, 109(4):399--405, 1985.

\bibitem{CernTeixWind09}
J.~\v{C}ern\'y, A. Teixeira, and D. Windisch.
\newblock Giant vacant component left by a random walk in a random $d$-regular graph.
\newblock {\em Preprint}.

\bibitem{Gord91}
D.M. Gordon.
\newblock Percolation in high dimensions.
\newblock {\em J. London Math. Soc.}, 44(2):373--384, 1991.

\bibitem{HaraSlad90}
T.~Hara and G.~Slade.
\newblock Mean-field critical phenomena for percolation in high dimensions.
\newblock {\em Commun. Math. Phys.}, 128:33--391, 1990.

\bibitem{Kest90}
H.~Kesten.
\newblock {\em Asymptotics in high dimensions for percolation}.
\newblock In: Disorder in pysical systems: {A} volume in honour of John
  Hammersley (G. Grimmett and D.J.A. Welsch eds.), 219-240, Clarendon Press,
  Oxford, 1990.

\bibitem{LiggSchoStac97}
T.~Liggett, R.H. Schonmann, and A.M. Stacey.
\newblock Domination by product measures.
\newblock {\em Ann. Probab.}, 25(1):71--95, 1997.

\bibitem{MadrSlad93}
N.~Madras and G.~Slade.
\newblock {\em The {S}elf-{A}voiding {W}alk}.
\newblock Birkh\"auser, Basel, 1993.

\bibitem{Mont56}
E.W. Montroll.
\newblock Random walks in multidimensional spaces, especially on periodic
  lattices.
\newblock {\em J. Soc. Industr. Appl. Math.}, 4(4):241--260, 1956.

\bibitem{Olve74}
F.W.J. Olver.
\newblock {\em Introduction to asymptotics and special functions}.
\newblock Academic Press, New York, 1974.

\bibitem{Pisz96}
A.~Pisztora.
\newblock Surface order large deviations for {I}sing, {P}otts and percolation
  models.
\newblock {\em Probab. Theory Relat. Fields}, 104:427--466, 1996.

\bibitem{SidoSzni09a}
V.~Sidoravicius and A.S. Sznitman.
\newblock Percolation for the vacant set of random interlacements.
\newblock {\em Comm. Pure Appl. Math.}, 62(6):831--858, 2009.


\bibitem{Szni07a}
A.S. Sznitman.
\newblock Vacant set of random interlacements and percolation.
\newblock {\em Ann. Math.}, in press, also available at arXiv:0704.2560.

\bibitem{Szni09a}
A.S. Sznitman.
\newblock Random walks on discrete cylinders and random interlacements.
\newblock {\em Probab. Theory Relat. Fields}, 145:143--174, 2009.

\bibitem{Szni08b}
A.S. Sznitman.
\newblock Upper bound on the disconnection time of discrete cylinders and
  random interlacements.
\newblock {\em Ann. Probab.}, 37(5):1715-1746, 2009.  
  
\bibitem{Szni09b}
A.S. Sznitman.
\newblock On the domination of random walk on a discrete cylinder by random
  interlacements.
\newblock {\em Electron. J. Probab.}, 14:1670-1704, 2009.

\bibitem{Szni10a}
A.S. Sznitman.
\newblock On the critical parameter of interlacement percolation in high dimension.
\newblock {\em Preprint}.
  
\bibitem{Teix08b}
A.~Teixeira.
\newblock Interlacement percolation on transient weighted graphs.
\newblock {\em Electron. J. Probab.}, 14:1604-1627, 2009.

\bibitem{Wind08}
D.~Windisch.
\newblock Random walk on a discrete torus and random interlacements.
\newblock {\em Electronic Communications in Probability}, 13:140--150, 2008.


\bibitem{Wind09}
D.~Windisch.
\newblock Random walks on discrete cylinders with large bases and random
  interlacements.
\newblock {\em Ann. Probab.}, in press, also available at arXiv:0907.1627. 

\end{thebibliography}
\end{document}